\documentclass[10pt]{amsart}
\usepackage{geometry}                % See geometry.pdf to learn the layout options. There are lots.
\usepackage{graphicx}
\usepackage{amssymb}
\usepackage{epstopdf}
\usepackage{subfigure}
\usepackage{amsmath}
\usepackage{amsthm}
\usepackage{amsfonts}
\usepackage{amssymb}
\usepackage{latexsym,hyperref}
\usepackage{subfigure,enumerate}

\newtheorem{thm}{Theorem}[section]

\theoremstyle{definition}

\theoremstyle{remark}
\newtheorem{rem}[thm]{Remark}
\numberwithin{equation}{section}

\newtheorem{remark}{Remark}[section]
\newtheorem*{maintheorem*}{Main Theorem}
\allowdisplaybreaks
\numberwithin{equation}{section}

\newcommand{\be} {\begin{equation}}
\newcommand{\ee} {\end{equation}}
\newcommand{\sbe} {\begin{subequations}}
\newcommand{\sen} {\end{subequations}}

\newcommand{\mcal}[1]{\mathcal{#1}}

\newcommand{\pa} {\partial}
\newcommand{\al} {\alpha}

\newcommand{\ga} {\gamma}
\newcommand{\ka} {\kappa}
\newcommand{\na} {\nabla}
\newcommand{\la} {\lambda}
\newcommand{\eps} {\varepsilon}

\newcommand{\R}{\mathbb{R}}

\newcommand{\Li}{\mathcal{L}}

\newcommand{\Dx}{\Delta x}
\newcommand{\Dy}{\Delta y}
\newcommand{\Dt}{\Delta t}

\newcommand{\z}{{\bf z}}
\newcommand{\q}{{\bf q}}
\newcommand{\uc}{{\bf u}}
\newcommand{\U}{{\bf U}}
\newcommand{\D}{{\bf D}}
\newcommand{\Po}{{\bf P}}
\newcommand{\W}{{\bf W}}

\newcommand{\F}{{\bf F}}
\newcommand{\sou}{{\bf s}}
\newcommand{\G}{{\bf G}}
\newcommand{\Q}{{\bf Q}}
\newcommand{\Hz}{{\bf H}}
\newcommand{\I}{{\bf I}}

\newcommand{\f}{{\bf f}}

\newcommand{\V}{{\bf V}}

\newcommand{\rhoa}{\rho_{\al}}
\newcommand{\va}{\mathbf{v}_{\al}}
\newcommand{\Ea}{E_{\al}}

\newcommand{\rhoi}{\rho_i}
\newcommand{\vi}{\mathbf{v}_i}
\newcommand{\Ei}{E_i}

\newcommand{\rhoe}{\rho_e}
\newcommand{\ve}{\mathbf{v}_e}
\newcommand{\Ee}{E_e}

\newcommand{\ea}{e_{\al}}
\newcommand{\eel}{e_e}
\newcommand{\emax}{e_m}
\newcommand{\ei}{e_i}

\newcommand{\qa}{{\bf q}_\al}

\newcommand{\E}{\mathbf{E}}
\newcommand{\B}{\mathbf{B}}

\newcommand{\rg}{\hat{r}_g}
\newcommand{\debye}{\hat{\la}_d}
\newcommand{\ch}{\hat{c}}

\newcommand{\aie}{\al\in\{i,e\}}

\newcommand{\ipmh}{{i\pm1/2}}
\newcommand{\iph}{{i+1/2}}
\newcommand{\jph}{{j+1/2}}
\newcommand{\imh}{{i-1/2}}
\newcommand{\jmh}{{j-1/2}}

\newcommand{\Unij}{\U_{i,j}^n}
\newcommand{\Uij}{\U_{i,j}}
\newcommand{\Sij}{{\bf S}_{i,j}}
\newcommand{\Sn}{{\bf S}}
\title{Entropy Stable Numerical Schemes for Two-Fluid plasma equations}
\author[Harish Kumar]{Harish Kumar}
\address[Harish Kumar] {\newline Seminar for Applied Mathematics (SAM) \newline
 Department of Mathematics, ETH Z\"urich, \newline Z\"urich, Switzerland}
\email[]{harish@math.ethz.ch}

\author[Siddhartha Mishra]{Siddhartha Mishra}
\address[Siddhartha Mishra]{\newline Seminar for Applied Mathematics (SAM) \newline
 Department of Mathematics, ETH Z\"urich, \newline Z\"urich, Switzerland}
 
\email[]{siddhartha.mishra@sam.math.ethz.ch}
\date{}                                           
\begin{document}
\maketitle
\begin{abstract}
Two-fluid ideal plasma equations are a generalized form of the ideal MHD equations in which electrons and ions are considered as separate species. The design of efficient numerical schemes for the these equations is complicated on account of their non-linear nature and the presence of stiff source terms, especially for high charge to mass ratios and for low Larmor radii. In this article, we design entropy stable finite difference schemes for the two-fluid equations by combining entropy conservative fluxes and suitable numerical diffusion operators. Furthermore, to overcome the time step restrictions imposed by the stiff source terms, we devise time-stepping routines based on implicit-explicit (IMEX)-Runge Kutta (RK) schemes. The special structure of the two-fluid plasma equations is exploited by us to design IMEX schemes in which only local (in each cell) linear equations need to be solved at each time step. Benchmark numerical experiments are presented to illustrate the robustness and accuracy of these schemes.
\end{abstract}
%---------------------------------------------

\section{Introduction}
\label{sec:intro}
An ensemble of plasma consists of ions, electrons and neutral particles. These particles interact through both short range (e.g.  collisions) and long range ( e.g. electromagnetic) forces. Plasmas are increasingly used in spacecraft propulsion, controlled nuclear fusion and in circuit breakers in the electrical power industry. 

Under the assumption of {\it quasi-neutrality} ( i.e. charge density difference between ions and electrons is neglected), the flow of plasmas is modeled by the ideal MHD equations (see \cite{book:goed}). Although, the ideal MHD equations have been successfully employed in modeling and simulating plasma flows, this model is derived by  ignoring the Hall effect and treating plasma flows as {\it single} fluid flows. These effects are very important for many applications, e.g. space plasmas, Hall current thrusters, field reversal configurations for magnetic plasma confinement and for fast magnetic reconnection.

In this article, we consider the more general ideal two-fluid model (see \cite{shumlak03},\cite{loverich},\cite{hakim06}) for collisionless plasmas. In the ideal two-fluid equations, electrons and ions are treated as different fluids by allowing them to posses different velocities and temperatures.  Assuming  {\it local thermodynamical equilibrium}, we write the two-fluid equations as a system of balance laws (see \cite{hakim06}):
\begin{equation}
\label{eq:bal_law}
\pa_t \uc +\mbox{div}\; \f(\uc) =\sou(\uc), \quad({\bf x},t) \in \R^3 \times (0,\infty).
\end{equation}
Here, $\uc = \uc(x,y,z,t)$ is the vector of non-dimensional conservative variables,
\be
\label{eq:u}
\uc=\left\{\rhoi,\rhoi\vi, \Ei, \rhoe,\rhoe\ve,\Ee,\B,\E,\phi,\psi \right\}^{\top}.
\ee
Here, the subscripts $\{i,e\}$ refer to the ion and electron species respectively, $\rho_{\{i,e\}}$ are the densities, $\mathbf{v}_{\{i,e\}}=(v_{\{i,e\}}^x,v_{\{i,e\}}^y,v_{\{i,e\}}^z)$ are the velocities, $E_{\{i,e\}}$ are the total energies, $\B=(B^x,B^y,B^z)$ is the magnetic field, $\E=(E^x,E^y,E^z)$ is the electric field and $\phi,\psi$ are the potentials. The flux vector $\f=(\f^x,\f^y,\f^z)$ can be written as,
\be
\label{eq:flux}
\f(\uc) =\left\{\begin{array}{c}
\f_i(\uc_i)\\
\f_e(\uc_e)\\
\f_m(\uc_m)
\end{array}\right\},\quad
\mbox{ where  }\quad
\f_\al(\uc_\al) =\left\{\begin{array}{c}
\rhoa \va\\
\rhoi \va\va^{\top} + p_\al\I \\
(\Ea+p_\al)\va\\
\end{array}\right\},
\ee
with $\aie$, and
\be
\label{eq:fs}
\f_m(\uc_m) =\left\{\begin{array}{c}
\mathcal{T}(\E)+\ka \psi\I \\
-\ch^2\mathcal{T}(\B) +\xi\ch^2\phi\I\\
\xi \E\\
\ka\ch^2\B
\end{array}\right\},\;\; 
\mbox{where}\;\;
\mathcal{T}({\bf K})=
\left[\begin{array}{ccc}0 & K^z & -K^y \\ -K^z& 0 & K^x \\K^y & -K^x & 0\end{array}\right],
\ee
for any vector ${\bf K}=(K^x,K^y,K^z).$ Here $\uc_\al=\{\rhoa,\rhoa\va,\Ea\}^{\top},\;\;\aie$, $\uc_m=\{\B,\E,\phi,\psi\}^{\top}$, $p_{\{i,e\}}$ are the pressures, $\xi,\kappa$ are penalizing speeds (see \cite{munz0}) and $\ch={c}/{v_i^T}$ is the normalized speed of light. Also, $v_i^T$ is the reference thermal velocity of ion. Writing the flux in component form (see \eqref{eq:flux},\eqref{eq:fs}), we observe that the first two components of the flux, $\f_\al(\uc_\al),\;\;\aie,$ are the nonlinear ion and electron Euler fluxes and the third component is the linear Maxwell flux. So, the homogeneous part of  \eqref{eq:bal_law} is hyperbolic.

The source term $\sou$ is given by,
\be
\sou(\uc) =\left\{\begin{array}{c}
0\\
\frac{1}{\rg}\rhoi(\E+\vi\times\B)\\
\frac{1}{\rg}\rhoi(\E\cdot\vi)\\
0\\
-\frac{\la_m}{\rg}\rhoe(\E+\ve\times\B)\\
-\frac{\la_m}{\rg}\rhoe(\E\cdot\ve)\\
{\bf 0}\\
-\frac{1}{\debye^2\rg}(r_i\rhoi \vi+ r_e\rhoe \ve)\\
\frac{\xi}{\debye^2\rg}(r_i\rhoi + r_e \rhoe)\\
0
\end{array}\right\}
\label{eq:flux_source},
\ee
with the charge to mass ratios $r_\al=q_\al/m_\al,\;\;\aie$ and the ion-electron mass ratio $\la_m={m_i}/{m_e}$. Two physically significant parameters appear in the source term namely, the normalized Larmor radius $\rg=\frac{r_g}{x_0}=\frac{m_iv_i^T}{q_iB_0 x_0}$ and the ion Debye length (normalized with respect to the Larmor radius) $\debye={\la_d}/{r_g}=\sqrt{\eps_0 v_i^{T ^2} / n_0 q_i}/r_g$ . Here,  $B_0$ is the reference magnetic field, $\eps_0$ is the permittivity of free space and $x_0$ is the reference length. The ion mass $m_i$ and ion charge $q_i$ are assumed to be 1. In addition, we assume that both the ion and the electron satisfies the ideal gas law:
\be
\label{eq:gas_law}
E_{\al}= \frac{p_{\al}}{\ga -1} + \frac{1}{2}\rhoa |\va|^2, \quad \aie,
\ee
with gas constant $\gamma=5/3$. In the above equations, we use the perfectly hyperbolic form of the Maxwell equation (see \cite{munz0}), which represent the evolution of magnetic field $\B$ and electric field $\E$.

The design of numerical schemes for systems of balance laws has undergone rapid development in the last two decades, see \cite{book:lev} for a detailed description of efficient schemes. The standard paradigm involves the use of finite volume (conservative finite difference) schemes in which the solution is evolved in terms of (approximate) solutions of Riemann problems at cell interfaces. Higher order accuracy in space is obtained by non-oscillatory interpolation procedures of the TVD, ENO and WENO types. An alternative is to use the Discontinuous Galerkin (DG) approach. High-order temporal accuracy results from strong stability preserving (SSP) Runge-Kutta (RK) methods. Source terms are included by using operator splitting approaches. 

Although the two-fluid equations are a system of balance laws, standard discretization techniques may fail to provide a robust approximation. Two major difficulties are present in the numerical analysis of the two-fluid equations: 1) the design of suitable numerical fluxes and 2) treatment of the source term that becomes increasingly stiffer as more realistic charge to mass ratios or more realistic Larmor radii (Debye lengths) are considered. 

Given the above challenges, very few robust numerical schemes exist for the two-fluid equations.   In \cite{shumlak03}, the authors derive a Roe-type Riemann solver. Time updates are performed by treating the stiff source term implicitly and the flux terms explicitly. The resulting non-linear equations are solved using Newton iterations. This method might be diffusive and may require many iterations for each time step. In \cite{hakim06}, the authors propose a wave propagation method (see \cite{book:lev}) for the spatial discretization. For time updates, a second -order operator splitting approach is used. A similar approach is taken in \cite{loverich,john}, where spatial discretization is based on discontinuous Galerkin (DG) methods and time update is based on SSP-RK methods. Both of these approaches are easy to implement but can be computationally expensive, especially for realistic charge to mass ratios.

Given the state of the art, we propose numerical schemes for the two-fluid equations with the following \emph{novel} features:
\begin{itemize}
\item First, we design entropy stable finite difference discretizations of the two fluid equations. The basis of our design is to ensure entropy stability as the two fluid equations satisfy an entropy inequality at the continuous level. We use the approach of \cite{tadmor1987} by constructing entropy conservative fluxes and suitable numerical diffusion operators to ensure entropy stability. Second-order entropy stable schemes are constructed following the framework of a recent paper \cite{ulrik}.

\item We discretize the source term in the two-fluid equations by an IMEX approach: the flux terms are discretized explicitly whereas the source term is discretized implicitly. The main feature of our schemes is their ability to use the special structure of the two-fluid equations that allows us to design IMEX schemes requiring the solution of only local (in each cell) linear equations at every time step. This is in contrast to the schemes proposed in \cite{shumlak03} that required the solution of non-linear iterations. The local equations that result from our approach can be solved exactly making our schemes computationally inexpensive. 
\end{itemize}

The rest of this article is organized as follows: In the following Section \ref{sec:cont}, we obtain an entropy estimate for the ideal two-fluid eqns. \eqref{eq:bal_law}. This result at the continuous level is then used to design an entropy stable finite difference scheme in Section \ref{sec:semi}. In Section \ref{sec:time}, we present IMEX-RK schemes for the temporal discretization. The resulting, algebraic system of equations is then solved exactly. In Section \ref{sec:num}, we simulate the nonlinear soliton propagation in the two-fluid plasma and a stiff Riemann problem to demonstrate the robustness and efficiency of these schemes.

%-----------------------------------------------------
\section{Analysis of Continuous Problem}
\label{sec:cont}
It is well known that solutions of \eqref{eq:bal_law} consists of discontinuities, even for smooth initial data. Hence, we need to consider the solutions of \eqref{eq:bal_law} in the weak sense. However, uniqueness of the solutions is still not guaranteed and we need to supplement  \eqref{eq:bal_law} with additional admissibility criteria to obtain a physically meaningful solution. This gives rise to concept of entropy. The standard thermodynamic entropies for ion and electron Euler flows are,
\be
\ea=-\frac{\rhoa s_\al}{\ga-1} \quad {\rm with } \;\;s_\al = \log(p_\al) - \gamma \log(\rhoa), \quad \aie.
\label{eq:ent_flow}
\ee
For the electromagnetic part we consider the {\it quadratic} entropy i.e electromagnetic energy, 
 \be
 e_m(\uc_m)=\frac{|\B|^2 + \phi^2}{2} + \frac{|\E|^2 + \psi^2}{2\ch^2}.
 \label{eq:ent_max}
 \ee 
We obtain the following entropy estimate,
%--------------------------------------------------------
 \begin{thm}
\label{thm:cont_ent_stability}
Let $\uc=\left\{\rhoi,\rhoi\vi, \Ei, \rhoe,\rhoe\ve,\Ee,\B,\E,\phi,\psi \right\}^{\top}$ be a weak solution of the two-fluid Eqns. \eqref{eq:bal_law} on $\R^3\times(0,\infty)$. Furthermore, assume that there exist constants $\rhoa^{\min},\rhoa^{\max}$ and $p^{\min}_\al$ such that,
$$
0 \leq \rhoa^{\min} \leq \rhoa \leq \rhoa^{\max}, \quad p_\al \geq p^{\min}_\al > 0, \quad \aie,
$$
then
\begin{equation}
\label{eq:ent_evo}
\frac{d}{dt}\int_{\R^3} (e_i +e_e + e_m)\;dx\;dy\;dz \leq C_1\int_{\R^3}  (\ei + \eel + e_m)\;dx\;dy\;dz  + C_2, 
\end{equation}
with constant $C_1$ and $C_2$ depending only on  $\rhoa^{\min},\rhoa^{\max}$, and $p^{\min}_\al$.
\end{thm}
\begin{proof} Let us first consider the fluid part of the equations. The entropy fluxes corresponding to the flow entropies \eqref{eq:ent_flow} are,
\be
\qa=-\frac{\rhoa s_\al\va}{\ga-1}=\va\ea, \quad \aie.
\label{eq:ent_flux_ent}
\ee
Assuming that $\uc$ is a smooth solution of \eqref{eq:bal_law}, the densities  $\rhoa$ and the pressures $p_\al,$ satisfy,
\begin{align*}
\pa_t \rhoa + \va \cdot\na \rhoa =0, \\
\pa_t {p_{\al}} + \gamma p_\al \na\cdot\va + \va\cdot\na p_{\al} =0.
\end{align*}
Using the expression for $s_\al,$ we get
$$
\pa_t s_\al + \va\cdot\na s_\al=0.
$$
Combining this with density advection we get {\it entropy conservation}, i.e.
\be
\pa_t e_\al + \na \cdot \qa =0.
\label{eq:flow_ent_cons}
\ee
Observe that \eqref{eq:flow_ent_cons} implies that the source term does not effect the evolution of fluid entropies. For weak solutions,  \eqref{eq:flow_ent_cons} reduces to {\it entropy inequality},
\be
\pa_t e_\al + \na \cdot \qa \le0.
\label{eq:flow_ent_inq}
\ee
Integrating over $\R^3$  and adding,
\be
\frac{d}{d t}\int_{\R^3} (e_i + e_e) dx \;dy\;dz  \le 0.
\label{eq:ent_decay}
\ee
For controlling the electromagnetic energy, we use the following inequality,
\begin{equation}
\label{eq:energy1}
\int_{\R^3} \left(\rhoa ^2 + |\rhoa \va |^2 + \Ea^2\right)\;\; dx\;dy\;dz  \leq C_3\int_{\R^3} \ea dx\;dy\;dz  + C_4,
\end{equation}
for some constants $C,\overline{C}$.The proof of \eqref{eq:energy1} is a simple consequence of the positivity of density and pressure and the use of the relative entropy method of Dafermos \cite{DAF1}. We multiply \eqref{eq:bal_law} with the vector,
$$
\left\{{\bf 0}_{10},\B,\frac{\E}{\ch^2},\phi, \frac{\psi}{\ch^2} \right\}^{\top}
$$
and note that flux terms are still in divergence form. Integrating over the whole space and using Cauchy's inequality on the right hand side, we get,
\be
\label{eq:em_ent_decay}
\frac{d}{dt} \int_{\R^3}e_m dx\;dy\;dz  \leq C_5 \left(\int_{\R^3} e_m dx\;dy\;dz  + \int_{\R^3} ( \ei + \eel) dx\;dy\;dz \right) + C_6.
\ee
Combining it with \eqref{eq:em_ent_decay} we obtain \eqref{eq:ent_evo}.
\end{proof}
\begin{rem}
Note that above proof of the theorem also gives a bound on the fluid energy of the system.
\end{rem}
%\subsection{Nature of the source terms}
%-----------------------------------------
\section{Semi-Discrete Schemes}
\label{sec:semi}
In the last section, we showed that solutions of the two-fluid equations satisfy the entropy estimate \eqref{eq:ent_evo}. In this section, we will design (semi-discrete) numerical schemes for the two-fluid equations that satisfy a discrete version of the entropy estimate. 

 For simplicity, we consider two-fluid eqns. \eqref{eq:bal_law} in two dimensions, i.e.,
\be
\pa_t \uc +\pa_x \f^x(\uc)+\pa_y \f^y(\uc)=\sou(\uc).
\label{eq:tf2d}
\ee
We discretize the two dimensional rectangular domain $D=(x_a,x_b)\times(y_a,y_b)$ uniformly with mesh size $(\Dx,\Dy)$. We define $x_i=x_a+i\Dx$ and $y_j=y_a+j\Dy$, $0\le i \le N_x,$ $0\le j\le N_y,$  such that $x_b=x_a+N_x\Dx$ and $y_b=y_a+N_y\Dx$. The domain is then divided into cells $I_{ij}=[x_\imh,x_\iph]\times [y_\jmh,y_\jph]$ with $x_\iph=\frac{x_i +x_{i+1}}{2}$ and  $y_\jph=\frac{y_j +y_{j+1}}{2}$. A standard semi-discrete finite difference scheme for the eqn. \eqref{eq:tf2d} can be written as,
\begin{equation}
\label{eq:semi_2d}
\frac{d \Uij}{dt}+ \frac{1}{\Dx}\left(\F_{\iph,j}^x -\F^{x}_{\imh,j}\right)+ \frac{1}{\Dy}\left(\F_{i,\jph}^y -\F^{y}_{i,\jmh}\right)  = \Sij(\U).
\end{equation}
Here, $\F^x_{\iph,j}$ and $\F^y_{i,\jph}$ are the numerical fluxes consistent with $\f^x$ and $\f^y$ respectively, and $\Sij(\U)=\sou(\Uij)$. We introduce the  {\it entropy variables} $\V$ and {\it entropy potential} $\chi^k$ which corresponds to the entropy $e=\{e_i,e_e,e_m\}^{\top}$
\be
\label{eq:ent_var_pot}
\V=\left\{\begin{array}{c}
\V_i\\
\V_e\\
\V_m
\end{array}\right\}=
\left\{\begin{array}{c}
\pa_{\uc_i}\ei(\uc_i)\\
\pa_{\uc_e}\eel(\uc_e)\\
\pa_{\uc_m}\emax(\uc_m)
\end{array}\right\},
\qquad
 \chi^k=\left\{\begin{array}{c}
\chi_i^k\\
\chi_e^k\\
\chi_m^k
\end{array}\right\}=
 \left\{\begin{array}{c}
\V_i^{\top}\f^k_i-\q^k_i\\
\V_e^{\top}\f^k_e-\q^k_e\\
\V_m^{\top}\f^k_m-\q^k_m
\end{array}\right\},
\ee
where $\q_m^k$ is the entropy flux for the Maxwell part corresponding to the entropy $e_m$ and $k\in\{x,y\}.$ 
We will follow the framework of Tadmor ( see \cite{tadmor1987,tadmor2004}) for designing an entropy stable scheme for the two-fluid equations. The first step is to design an entropy conservative flux.
 \subsection{Entropy conservative flux}
\label{subsec:ec} We require the following notation:
\begin{eqnarray*}
[a]_{\iph,j}=a_{i+1,j}-a_{i,j},\quad \overline{a}_{\iph,j}=\frac{1}{2}(a_{i+1,j}+a_{i,j}),\\ \nonumber
[a]_{i,\jph}=a_{i,j+1}-a_{i,j},\quad \overline{a}_{i,\jph}=\frac{1}{2}(a_{i,j+1}+a_{i,j}).
\end{eqnarray*}
Following \cite{tadmor1987}, an entropy conservative flux $\hat{\F} = \{\hat{\F}^x,\hat{\F}^y\}$ is defined as a consistent flux that satisfies
\be
\label{eq:ecflux}
%\begin{aligned}
[\V]_{\iph,j}^{\top} \hat{\F}^x_{\iph,j} =[\chi^x]_{\iph,j}, \qquad [\V]_{i,\jph}^{\top} \hat{\F}^y_{i,\jph} =[\chi^y]_{i,\jph}.
%\end{aligned}
\ee
%\end{thm} 
In general, the relation for conservative flux, \eqref{eq:ecflux} provides one equation for several unknowns. Hence, entropy conservative numerical flux is not unique. We will now describe entropy conservative numerical fluxes for the fluid part of the two-fluid equations.

In \cite{roe}, Ismail and Roe have derived an expression for entropy conservative numerical fluxes for Euler equations of gas dynamics. As the fluid part of \eqref{eq:bal_law} consists of two independent Euler fluxes, we can use the expression derived in \cite{roe} for the entropy conservative numerical flux of the Euler flows of ion and electron. We need to introduce parametric vectors $\z_\al,\;\;\aie$,
\begin{equation}
\label{eq:ent_var_f}
\z_\al=\left[\begin{array}{c}
z_\al^1\\
z_\al^2\\
z_\al^3\\
z_\al^4\\
z_\al^5
\end{array}\right]\quad
=\;\; \sqrt{\frac{\rhoa}{p_\al}}\left[\begin{array}{c}
1\\
v_\al^x\\
v_\al^y\\
v_\al^z\\
p_\al
\end{array}\right] , \qquad
\aie.
\end{equation} 
Then the entropy conservative numerical flux in x-direction is given by  $\hat{\F}^x_{\al,\iph,j}=[\hat{\F}^{x,1}_{\al,\iph,,j}, \hat{\F}^{x,2}_{\al,\iph,j}$, $ \hat{\F}^{x,3}_{\al,\iph,j}, \hat{\F}^{x,4}_{\al,\iph,j}, \hat{\F}^{x,5}_{\al,\iph,j}]^{\top}$, with,
\begin{eqnarray}
\hat{\F}^{x,1}_{\al,\iph,j}&=&\overline{z^2}_{\al,\iph,j}{z_{\al}^5}^{\ln}_{\iph,j},\\ \nonumber
\hat{\F}^{x,2}_{\al,\iph,j}&=& m^5_{\al,\iph,j}+ m^2_{\al,\iph,j}\hat{\F}^{x,1}_{\al,\iph,j} ,\\ \nonumber
\hat{\F}^{x,3}_{\al,\iph,j}&=&m^3_{\al,\iph,j}\hat{\F}^{x,1}_{\al,\iph,j} ,\\ \nonumber
\hat{\F}^{x,4}_{\al,\iph,j}&=&m^4_{\al,\iph,j}\hat{\F}^{x,1}_{\al,\iph,j} ,\\ \nonumber
\hat{\F}^{x,5}_{\al,\iph,j}&=&\frac{1}{2 \overline{z^1}_{\al,\iph,j} } \left( \frac{\gamma+1}{\ga-1} \frac{\hat{\F}^{x,1}_{\al,\iph,j}} 
 { {z^1}^{\ln}_{\al,\iph,j}}  +  \overline{z^2}_{\al,\iph,j} \hat{\F}^{x,2}_{\al,\iph,j}\right.  \\ \nonumber
 &&\left.+   \overline{z^3}_{\al,\iph} \hat{\F}^{x,3}_{\al,\iph,j} +   \overline{z^4}_{\al,\iph,j} \hat{\F}^{x,4}_{\al,\iph,j}\right).
\end{eqnarray}
and entropy conservative numerical flux in y-direction is,
$\hat{\F}^y_{\al,i,\jph}=[\hat{\F}^{y,1}_{\al,i,\jph}, \hat{\F}^{y,2}_{\al,i,\jph}$, $ \hat{\F}^{y,3}_{\al,i,\jph}, \hat{\F}^{y,4}_{\al,i,\jph}, \hat{\F}^{y,5}_{\al,i,\jph}]^{\top}$, with,
\begin{eqnarray}
\hat{\F}^{y,1}_{\al,i,\jph}&=&\overline{z^3}_{\al,i,\jph}{z^5}^{\ln}_{\al,i,\jph},\\ \nonumber
\hat{\F}^{y,2}_{\al,i,\jph}&=&m^2_{\al,i,\jph}  \hat{\F}^{y,1}_{\al,i,\jph} ,\\ \nonumber
\hat{\F}^{y,3}_{\al,i,\jph}&=&m^3_{\al,i,\jph}+ m^3_{\al,i,\jph} \hat{\F}^{y,1}_{\al,i,\jph} ,\\ \nonumber
\hat{\F}^{y,4}_{\al,i,\jph}&=&m^4_{\al,i,\jph}  \hat{\F}^{y,1}_{\al,i,\jph},\\ \nonumber
\hat{\F}^{y,5}_{\al,i,\jph}&=&\frac{1}{2 \overline{z^1}_{\al,i,\jph} } \left( \frac{\gamma+1}{\ga-1} \frac{\hat{\F}^{y,1}_{\al,i,\jph}} 
 { {z_{\al,}^1}^{\ln}_{i,\jph}}  +  \overline{z^2}_{\al,i,\jph}\hat{\F}^{y,2}_{\al,i,\jph}\right.  \\ \nonumber
 &&\left.+   \overline{z^3}_{\al,i,\jph} \hat{\F}^{y,3}_{\al,i,\jph} +   \overline{z^4}_{\al,i,\jph} \hat{\F}^{y,4}_{\al,i,\jph}\right),
\end{eqnarray}
Here, $a_{\iph,j}^{ln}$ and $a_{i,\jph}^{ln}$ denotes the logarithmic means defined as,
$$
a^{ln}_{\iph,j} = \frac{[a]_{\iph,j}}{[\log{(a)}]_{\iph,j}}, \qquad a^{ln}_{i,\jph} = \frac{[a]_{i,\jph}}{[\log{(a)}]_{i,\jph}},
$$
and
$$
m^r_{\al,\iph,j}=\frac{\overline{z^r}_{\al,\iph,j}}{\overline{z^1}_{\al,\iph,j}}, \qquad m^r_{\al,i,\jph}=\frac{\overline{z^r}_{\al,i,\jph}}{\overline{z^1}_{\al,i,\jph}},\qquad \mbox{for}\qquad r\in\{2,3,4,5\}.
$$
Now we will consider the electromagnetic part. Note the Maxwell flux is linear.
%The entropy variable and entropy potential are
%\begin{equation}
%\V_{m}(\uc_m)= \uc_m,\quad \chi(\uc_m)= \frac{1}{2}\uc_m^{\top} \uc_m.
%\label{eq:max_ent_var_pot}
%\end{equation}
Then, it is easy to check that the entropy conservative numerical flux for the electromagnetic part is
\begin{equation}
\label{eq:ent_con_flux_max}
\hat{\F}^x_{m,\iph,j}=\frac{1}{2} (\f^x(\U_{m,i,j})+ \f^x(\U_{m,i+1,j}) ),\quad \hat{\F}^y_{m,i,\jph}=\frac{1}{2} (\f^y(\U_{m,i,j})+ \f^y(\U_{m,i,j+1}) ).
\end{equation}
Combining all the parts, the entropy conservative numerical flux for the Eqn. \eqref{eq:bal_law} are given by,
 \be
 \label{eq:ent_con_flux}
\hat{\F}^x_{\iph,j}=\left\{\begin{array}{c}
\hat{\F}^x_{i,\iph,j}\\
\hat{\F}^x_{e,\iph,j}\\
\hat{\F}^x_{m,\iph,j}\\
\end{array}\right\},\qquad
\hat{\F}^y_{i,\jph}=\left\{\begin{array}{c}
\hat{\F}^y_{i,\jph}\\
\hat{\F}^y_{e,i,\jph}\\
\hat{\F}^y_{m,i,\jph}\\
\end{array}\right\}.
\ee
%-----------------------------------------------------
\subsection{Numerical diffusion operator}
\label{subsec:diff}
As entropy is dissipated at shocks, we need to add {\it entropy stable} numerical diffusion operators to avoid spurious oscillations at shocks. Following \cite{tadmor2004},the resulting numerical fluxes are of the form,
\begin{equation}
\label{eq:ent_stable_flux}
\F^x_{\iph,j}=\hat{\F}^x_{\iph}-\frac{1}{2}\D^x_{\iph}[\V]_{\iph,j}, \qquad \F^y_{i,\jph}=\hat{\F}^x_{\iph}-\frac{1}{2}\D^y_{i,\jph}[\V]_{i,\jph}.
\end{equation}
with diffusion matrices are given by,
\begin{equation}
\label{eq:diff_matrix}
\D^x_{\iph}=R^x_{\iph,j}\Lambda^x_{\iph,j} R^{x\top}_{\iph,j},\qquad \D^y_{i,\jph}=R_{i,\jph}^y\Lambda^y_{i,\jph} R^{y \top}_{i,\jph}.
\end{equation}
Here $R^{\{x,y\}}$ are the matrices of right eigenvectors of Jacobians $\pa_{\uc}\f^{\{x,y\}}$ and $\Lambda^{\{x,y\}}$ are diagonal matrices of eigenvalues in the $x$- and $y$-directions, respectively. We will use a Rusanov  type diffusion operator given by a $18\times 18$ matrix,
$$
\Lambda^{\{x,y\}}=\Lambda=\mbox{diag}\{(\max_{1\le i\le5}|\la^x_i|) I_{5\times5}, (\max_{6\le i\le10}|\la^x_i| )I_{5\times5},(\max_{1\le i\le18}|\la^x_i|) I_{8\times8}\}.
$$ 
We use the eigenvector scaling due to Barth \cite{barth} for defining the eigenvector matrices.
 
%-------------------------------------------------------------------
\subsection{Second Order Dissipation Operator} The diffusion operators described above are of first order, as the jump term $[\V]$ is of order $\Dx$. To obtain the second order accurate scheme, we can perform piecewise linear reconstructions of the entropy variable $\V$. However, a straightforward reconstruction of the entropy variables may not be entropy stable. In \cite{ulrik}, the authors have constructed entropy stable second order (and even higher-order) diffusion operators. For simplicity, we will consider the diffusion operator, $\D^x_{\iph,j}[\V]_{\iph,j}$ in x-direction only. The diffusion operator in y-direction, $\D^y_{i,\jph}[\V]_{i,\jph}$ can be constructed analogously. We need to introduce \emph{scaled} entropy variables,
$$
\W^{x,\pm}_{i,j}=R_{\ipmh,j}^{x \top}\V_{i,j}.
$$
If $ \tilde{\W}^{x,\pm}_{i,j}$ are the reconstructed values of $\W^{x \pm}$ in the x-direction, then the corresponding reconstructed values $\Po_i^{x\pm}$ for $\V_{ij}$ are given by,
$$
\Po_{ij}^{x\pm}=\{R^{x T}_{i\pm\iph,j}\}^{(-1)} \tilde{\W}^{x,\pm}_{i,j}.
$$
The resulting second order entropy stable flux is then given by,
\be
\label{eq:ent_2nd_flux}
\F^x_{\iph,j}=\hat{\F}^x_{\iph}-\frac{1}{2}\D^x_{\iph}[\Po^x]_{\iph,j}, 
\ee
where the jump term $[\Po^x]_{\iph,j}$ is given by,
$$
[\Po^x]_{\iph,j}=\Po_{i+1,j}^{x-} - \Po_{i,j}^{x+}.
$$
A sufficient condition for the scheme to be entropy stable (see \cite{ulrik}) is the existence of a diagonal matrix $B^x\ge 0$, such that,
$$
 [\tilde{\W}^x]_{\iph,j}=B^x_{\iph,j} [\W^x]_{\iph,j},
$$
i.e. the reconstruction of $\W^x$ has to satisfy a {\it sign preserving property} along the interfaces of each cell. Component-wise this can be written as,
\be
\label{eq:sign_comp}
\mbox{sign}([\tilde{w}^l])=\mbox{sign}([w^l]),
\ee
for each component $w^l$ and $\tilde{w}^l$ of  $\W^x$ and $\tilde{\W}^x$, respectively.
\subsection{Reconstruction Procedure} We suppress the $j$-dependence below for notational convenience. The reconstruction for $\W^x$ is performed component-wise, so that \eqref{eq:sign_comp} is satisfied. Let us define jump of component $w$ of the variable $\W^x$, 
\be
\delta_{\iph}= [w]_{\iph}.
\label{eq:difference}
\ee
Consider $\phi$, a slope limiter satisfying $\phi(\theta^{-1})=\phi(\theta)\theta^{-1}$ and define divided differences,
$$
\theta_i^-=\frac{\delta_{\iph}}{\delta_{\imh}} \quad \mbox{and}\quad\theta_i^+= \frac{\delta_\imh}{\delta_\iph}.
$$
Then the reconstructed values of $w$ in the cell $I_i$ are
$$
\tilde{w}_i^-=w_i^--\frac{1}{2}\phi(\theta_i^-)\delta_\imh, \qquad \tilde{w}_i^+=w_i^++\frac{1}{2}\phi(\theta_{i+1}^+)\delta_\iph.
$$
Using these values we obtain
$$
[ \tilde{w}]_\iph=\left( 1- \frac{1}{2} ( \phi(\theta_i^+) + \phi(\theta_{i+1}^-)   )\right)\delta_{\iph}.
$$
This shows that the sign property is satisfied iff
$$
\phi(\theta)\le 1, \quad \forall \;\;\theta\in\R.
$$
In this article, we will use the minmod limiter based reconstruction which satisfies the sign preserving property (see \cite{ulrik}). The minmod limiter is given by,
 \be
\phi(\theta) =\begin{cases}
0, & \mbox{if} \;\;\theta<0,\\
\theta, & \mbox{if} \;\;0\le\theta\le1,\\
1, & \mbox{else}.
\end{cases}
\ee

%%-----------------------------------------------------

\subsection{Discrete entropy stability} In this section, we prove that scheme given by the flux  \eqref{eq:ent_2nd_flux} is entropy stable i.e. it satisfies a discrete version of the entropy estimate \eqref{eq:ent_evo}. We have the following result,
 \label{subsec:dis_es}
\begin{thm}
\label{theo:2}
The semi-discrete finite difference scheme \eqref{eq:semi_2d}, with entropy stable numerical flux \eqref{eq:ent_2nd_flux}, is second order accurate for smooth solutions. Furthermore, it satisfies,
\be
\frac{d}{dt}\sum_{i,j}(e_{i,i,j}+e_{e,i,j} + e_{m,i,j}) \Dx\Dy \le C_7 \sum_{i,j}(e_{i,i,j}+e_{e,i,j} + e_{m,i,j}) \Dx\Dy+C_8
\label{eq:dis_ener_stab}
\ee
if conditions for Theorem \ref{thm:cont_ent_stability} are satisfied.
%\be
%\frac{d e_\al}{ dt}+\frac{1}{\Dx}\left(\Q^x_{\al,\iph,j} -  \Q^x_{\al,\imh,j} \right) +\frac{1}{\Dy}\left( \Q^y_{\al,i,\jph} - {\Q}^y_{\al,i,\jmh} \right)\le0, \quad \aie,
%\label{eq:es2d_fluid}
%%\ee
%for the numerical entropy flux,
%\be
%\label{eq:ent_stab_ent_flux}
%\Q^x_{\iph,j}=\overline{\V}_{\iph,j}^{\top}\F^x_{\iph,j}-\bar{\chi}_{\iph,j}, \qquad \Q^y_{i,\jph}=\overline{\V}_{i,\jph}^{\top}\F^y_{i,\jph}-\bar{\chi}_{i,\jph}.
%\ee

\end{thm}
\begin{proof}
It is easy to see that the scheme is of second order accuracy, as both the entropy conservative flux $\hat{\F}$ and the numerical diffusion operator, are second order accurate for smooth solutions. Now, consider the fluid part of \eqref{eq:semi_2d}, i.e.
\begin{equation}
\frac{d \U_{\al,i,j}}{dt}+ \frac{1}{\Dx}\left(\F_{\al,\iph,j}^x -\F^{x}_{\al,\imh,j}\right)+ \frac{1}{\Dy}\left(\F_{\al,i,\jph}^y -\F^{y}_{\al,i,\jmh}\right)  = \Sn_{\al,i,j}(\U),
\label{eq:semi_dis_f}
\end{equation}
for $\aie$ with {\it entropy numerical fluxes},
\be
\label{eq:ent_stab_ent_flux}
\Q^x_{\iph,j}=\overline{\V}_{\iph,j}^{\top}\F^x_{\iph,j}-\bar{\chi}_{\iph,j}, \qquad \Q^y_{i,\jph}=\overline{\V}_{i,\jph}^{\top}\F^y_{i,\jph}-\bar{\chi}_{i,\jph}.
\ee
Multiplying \eqref{eq:semi_dis_f} with $\V_{\al,i,j}^{\top}$ and imitating the proof of Theorem 2.2 from \cite{tadmor1987}, we get
\begin{eqnarray*}
\frac{d e_\al(\Uij)}{ dt} &=&\frac{1}{\Dx}\left( \hat{\Q}^x_{\iph,j} -\hat{\Q}^x_{\imh,j}\right)-\frac{1}{\Dx}\left( \hat{\Q}^y_{i,\jph} -\hat{\Q}^y_{i,\jmh}\right)+\V^{\top}_{\al,i,j}\Sn_{\al,i,j}(\U)\\
&-&\frac{1}{2 \Dx}\left( [\V]_{\iph,j}^{\top} \D^x_{\iph,j} [\Po^x]_{\iph,j} +  [\V]_{\imh,j}^{\top} \D^x_{\imh,j} [\Po^x]_{\imh,j} \right)\\
&-&\frac{1}{2 \Dy}\left( [\V]_{i,\jph}^{\top} \D^y_{i,\jph} [\Po^y]_{i,\jph} +  [\V]_{i,\jmh}^{\top} \D^y_{i,\jmh} [\Po^y]_{i,\jmh} \right) \\
&=&-\frac{1}{\Dx}\left( {\Q}^x_{\iph,j} -{\Q}^x_{\imh,j}\right)-\frac{1}{\Dx}\left( {\Q}^y_{i,\jph} -{\Q}^y_{i,\jmh}\right) +\V^{\top}_{\al,i,j}\Sn_{\al,i,j}(\U)\\
&-&\frac{1}{4 \Dx}\left( [\V]_{\iph,j}^{\top} \D^x_{\iph,j} [\Po^x]_{\iph,j} +  [\V]_{\imh,j}^{\top} \D^x_{\imh,j} [\Po^x]_{\imh,j} \right)\\
&-&\frac{1}{4 \Dy}\left( [\V]_{i,\jph}^{\top} \D^y_{i,\jph} [\Po^y]_{i,\jph} +  [\V]_{i,\jmh}^{\top} \D^y_{i,\jmh} [\Po^y]_{i,\jmh} \right) 
\end{eqnarray*}
Here 
$$
\hat{\Q}^x_{\iph,j}=\overline{\V}_{\iph,j}^{\top}\hat{\F}^x_{\iph,j}-\bar{\chi}_{\iph,j}, \;\mbox{and}\; \;\hat{\Q}^y_{i,\jph}=\overline{\V}_{i,\jph}^{\top}\hat{\F}^y_{i,\jph}-\bar{\chi}_{i,\jph}.
$$
are entropy fluxes corresponding to the entropy conservative fluxes $\hat{\F}^x$ and $\hat{\F}^y$ respectively. Let us consider the diffusion terms. Ignoring all the indices, each diffusion term satisfies,
\begin{eqnarray*}
\label{eq:recons_scaled_entropy_var}
[\V]^\top\D[\Po]&=&[\V]^\top R\Lambda R^\top [\Po]\\
&=&[\V]^\top R\Lambda R^\top (R^\top)^{(-1)}[\tilde{\W}]\\ 
&=&(R^\top[\V])^{\top} \Lambda B ([\W])\\ 
&=&(R^\top[\V])^{\top} \Lambda B (R^\top \V)\\ 
&\geq &0,
\end{eqnarray*}
as both $B$ and $\Lambda$ are non-negative diagonal matrices. So, we get
$$
\frac{d e_{\al,i,j}}{ dt}+\frac{1}{\Dx}\left(\Q^x_{\al,\iph,j} -  \Q^x_{\al,\imh,j} \right) +\frac{1}{\Dy}\left( \Q^y_{\al,i,\jph} - {\Q}^y_{\al,i,\jmh} \right)\le\V^{\top}_{\al,i,j}\Sn_{\al,i,j}.
$$
A simple calculation shows that,
$$
\V^{\top}_{\al,i,j}\Sn_{\al,i,j}=0.
$$
This results in the fluid entropy inequality,
\be
\frac{d e_{\al,i,j}}{dt}+\frac{1}{\Dx}\left(\Q^x_{\al,\iph,j} -  \Q^x_{\al,\imh,j} \right) +\frac{1}{\Dy}\left( \Q^y_{\al,i,\jph} - {\Q}^y_{\al,i,\jmh} \right)\le0, \quad \aie,
\label{eq:es2d_fluid}
\ee
summing over all the cells we get,
\be
\frac{d}{dt}\sum_{i,j}e_{\al,i,j}\Dx\Dy\le0, \quad \aie,
\label{eq:dis_fluid_ent_inq}
\ee

Repeating the entropy argument of Dafermos \cite{DAF1} used in Theorem \ref{thm:cont_ent_stability} we get an discrete energy estimate for fluid part,
\begin{equation}
\label{eq:sid_energy_f_inq}
\sum_{i,j} \left(\rho_{\al,i,j} ^2 + |\rho_{\al,i,j} {\bf v}_{\al,i,j} |^2 + E_{\al,i,j}^2\right)\;\; \Dx\Dy \leq C_9\sum_{i,j}e_{\al,i,j}\Dx\Dy  + C_{10},
\end{equation}
Imitating the proof of Theorem \ref{thm:cont_ent_stability} where integration is replaced by summation, we get,
\be
\label{eq:em_ent_decay}
\frac{d}{dt} \sum_{i,j} e_{m,i,j} \Dx\Dy  \leq C_{11} \sum_{i,j} (e_{m,i,j} + e_{i,i,j} +e_{e,i,j}) \Dx\Dy   + C_{12}.
\ee
 Combining with \eqref{eq:dis_fluid_ent_inq}, we get \eqref{eq:dis_ener_stab}.
\end{proof}
\begin{remark}
In theorem \ref{theo:2}, the discrete energy estimate \eqref{eq:dis_ener_stab} is satisfied only if the electron and ion density and pressure (as required by theorem \ref{thm:cont_ent_stability}) are positive. We assume that this positivity holds for the scheme. Currently, it is not possible to prove that this positivity is also a consequence of the numerical scheme. However, we expect that the use of positivity preserving limiters (like those in \cite{Sp}) will enable us to prove positivity. 
\end{remark}

%%-----------------------------------------
\section{Fully Discrete Schemes}
\label{sec:time}
Let $\U^n$ is the discrete solution at $t^n$, and $\Dt=t^{n+1}-t^n$. Then a semi-discrete scheme \eqref{eq:semi_2d} can be written as,
\be
\frac{d \Unij}{dt}= \Li_{i,j} (\U^n)+ \Sij(\U^n),
\label{eq:time_ode}
\ee
where, 
$$
\Li_{i,j}(\U^n)= -\frac{1}{\Dx}\left(\F_{\iph,j}^x -\F^{x}_{\imh,j}\right)- \frac{1}{\Dy}\left(\F_{i,\jph}^y -\F^{y}_{i,\jmh}\right), \;\;\mbox{and}\;\; \Sij(\U^n)=\sou(\Unij).
 $$
We describe two different time discretization schemes below.
\subsection{Explicit Schemes}
\label{subsec:exp}
We use explicit Runge-Kutta (RK) time marching schemes for the time-discretizing of the two-fluid equations.  For simplicity, we restrict ourselves to the second- and third-order accurate RK schemes (see \cite{tvd}). These methods are strong stability preserving (SSP). In order to advance a numerical solution from time $t^n$ to $t^{n+1}$, the SSP-RK algorithm is as follows:
\begin{enumerate}
\item[1.] Set $\U^{(0)}=\U^{n}$.
\item[2.] For $m=1, .... ,k+1$, compute,
$$
\U^{(m)}_{i,j} = \sum_{l=0}^{m-1}\alpha_{ml} \U^{(l)}_{i,j} + \beta_{ml} \Dt^n ( \Li_{i,j}( \U^{(l)}) + \Sij( \U^{(l)} )).
$$
\item[3.] Set $\U^{n+1}_{i,j}=\U_{i,j}^{(k+1)}$.
\end{enumerate}
The coefficients $\al_{ml}$ and $\beta_{ml}$ are given in Table \ref{tab:rk_coeff}. 
\begin{table}
\begin{center} \begin{tabular}{c|cccccccc}
order 	&			& $\alpha_{il}$	&		&		&		& $\beta_{il}$\\
\hline\\
2 	&		1 	&		&		&		&	1\\
 	&		1/2 	&	1/2	&		&		&	0	&	1/2\\
\hline\\
3	&		1 	&		& 		&		&	1\\	
	&		3/4 	&	1/4	& 		&		&	0 	&	1/4\\	
	&		1/3 	&	0	& 	2/3	&		&	0 	&	0 	&	2/3
\end{tabular}\end{center}	
\caption{Parameters for Runge-Kutta time marching schemes.}
\label{tab:rk_coeff}
\end{table}
%----------------------------------------------------------------
\subsection{IMEX-RK Schemes}
\label{subsec:imex}
As discussed in section \ref{sec:intro}, two-fluid equations contain the following parameters: the speed of light, mass ratio of ions to electrons, Debye length, and the Larmor radius. These parameters determine the time scales of the flow and may impose severe restrictions on the time step of explicit time marching schemes. Hence, we consider IMEX methods in this section. An Implicit-Explicit Runge-Kutta (IMEX-RK) scheme for \eqref{eq:bal_law}, is based on the implicit treatment of the stiff source term and an explicit treatment of the convective flux terms. This allows us to overcome stiffness due to the source terms. 

We will use SSP-RK schemes, as described above, with each intermediate Euler update being carried out by solving,
\begin{eqnarray}
\label{eq:imex_step}
\Uij^{m+1}&=&\Uij^m+\Dt  \mcal{L}_{i,j}(\U^{m}) +\Dt \Sij(\U^{m+1}),
\end{eqnarray} 
for $\U^{m+1}$. Usually \eqref{eq:imex_step} is solved using some iterative methods. However, we can exploit the special  structure of the source term for the two-fluid equations to solve \eqref{eq:imex_step} \emph{exactly}. We proceed as follows: \\
Denote $\U=\{\W_1,\W_2,\W_3\}$ with, 
\begin{eqnarray*}
\W_1& = & \{\rhoi,\rhoe,B^x,B^y,B^z,\psi\}^{\top} \\
\W_2& = & \{\rhoi v_i^x,\rhoi v_i^y,\rhoi v_i^z,\rhoe v_e^x,\rhoe v_e^y,\rhoe v_e^z,E^x,E^y,E^z\}^{\top} \\
\W_3& = & \{\Ei,\Ee,\phi\}^{\top} 
\end{eqnarray*}
We observe that \eqref{eq:imex_step} can be rewritten in the following three blocks,
\sbe
\label{eq:IMEX2}
\begin{eqnarray}
\label{eq:IMEX2a}
\W_{1}^{(m+1)} &=&\G_1(\U^{(m)}),\\
\label{eq:IMEX2b}
\W_{2}^{(m+1)} &=&\G_2(\U^{(m)})+ \mathcal{A}(\W_{1}^{(m+1)})\W_{2}^{(m)},\\
\label{eq:IMEX2c}
\W_{3}^{(m+1)} &=&\G_3(\U^{(m)})+\Hz(\W_{1}^{(m+1)},\W_{2}^{(m+1)}).
\end{eqnarray}
\sen
Here $\G_1,\G_2$ and $\G_3$ are the explicit  parts of  \eqref{eq:imex_step} for the variables $\W_1,\W_2$ and $\W_3$ respectively. The Eqns. \eqref{eq:IMEX2} are then solved in sequential manner:
\begin{enumerate}[I)]
\item  Equation \eqref{eq:IMEX2a} is updated explicitly, as it involves the evaluation of the known terms from the previous time step. 
 \item Note that the matrix $\mathcal{A}(\W_{1}^{(m+1)})$ in Eqn. \eqref{eq:IMEX2b} is,
{\tiny \be
 \label{eq:A}
\left[\begin{array}{ccccccccc}
0 & \frac{B^{z,(m+1)}}{\hat{r}_g} & -\frac{B^{y,(m+1)}}{\hat{r}_g} & 0 & 0 & 0 & \frac{\rho_i^{(m+1)}}{\hat{r}_g} & 0 & 0\\
-\frac{B^{z,(m+1)}}{\hat{r}_g} & 0 & \frac{B^{x,(m+1)}}{\hat{r}_g} &  0 & 0 & 0 & 0 &  \frac{\rho_i^{(m+1)}}{\hat{r}_g}   & 0\\
 \frac{B^{y,(m+1)}}{\hat{r}_g} & -\frac{B^{x,(m+1)}}{\hat{r}_g} & 0 & 0 & 0 &  0 & 0 & 0 & \frac{\rho_i^{(m+1)}}{\hat{r}_g} \\
0 & 0 & 0 & 0 & \frac{B^{z,(m+1)}}{\hat{r}_{e,g}} & -\frac{B^{y,(m+1)}}{\hat{r}_{e,g}} & \frac{\rho_e^{(m+1)}}{\hat{r}_{e,g}} & 0 & 0\\
0 & 0 & 0 & -\frac{B^{z,(m+1)}}{\hat{r}_{e,g}} & 0 &\frac{B^{x,(m+1)}}{\hat{r}_{e,g}} &  0 & \frac{\rho_e^{(m+1)}}{\hat{r}_{e,g}}  & 0\\
0 & 0 & 0 &   \frac{B^{y,(m+1)}}{\hat{r}_{e,g}} & -\frac{B^{x,(m+1)}}{\hat{r}_{e,g}} & 0 &  0 & 0 &\frac{\rho_e^{(m+1)}}{\hat{r}_{e,g}} \\
 \frac{-r_i}{K} & 0 & 0 & \frac{-r_e}{K} & 0 & 0 & 0 & 0 & 0   \\
0 & \frac{-r_i}{K}  & 0 &  0 &\frac{-r_e}{K}  & 0 & 0 & 0 & 0   \\  
 0 & 0 &\frac{-r_i}{K}  & 0 & 0 &\frac{-r_e}{K}   & 0 & 0 & 0 
\end{array}\right]
\ee}
with $\hat{r}_{e,g} =-\hat{r}_g/\lambda_m$ and $K=\hat{\lambda}^2\hat{r}_g$. All the quantities in the matrix are already computed in step I. So, we can rewrite Eqn. \eqref{eq:IMEX2b} as,
\begin{equation}
\label{eq:IMEX_solve}
  \W_{2}^{(m+1)} =\left(\mathbf{I}- (\Dt) \mathcal{A}(\W_{1}^{(m+1)}) \right)^{(-1)}\G_2(\U^{(m)}).
\end{equation}
which can evaluated exactly.

\item The Eqn. \eqref{eq:IMEX2c} is now updated  for $\W_3^{m+1}$ by evaluating $\Hz(\W_{1}^{m+1},\W_{2}^{m+1})$.
\end{enumerate}

\begin{remark}
The IMEX scheme proposed above does not require any non-linear Newton solves or any global matrix inversions. It only needs explicit evaluations of the inverse of a local $9 \times 9$ matrix in each cell making this scheme computationally inexpensive. Furthermore, there are \emph{no} local linearizations or approximations being used in the scheme. It uses an exact solution of the time stepping update \eqref{eq:imex_step}.
\end{remark}

\begin{remark}
Note that the wave speeds of the system depend on the speed of light and the sound speeds of the electron and ion. The speed of these waves is either specified or determined by the flux terms of the two-fluid equations. Consequently, an explicit in time, evaluation of the flux terms, as in an IMEX scheme, might still lead to severe time step restrictions on account of these waves. 
\end{remark}

%%--------------------------------------------------------------------------
\section{Numerical Results}
\label{sec:num}
%%--------------------------------------------------------------------------

We present a set of numerical experiments to demonstrate the robustness of the proposed schemes.

\subsection{Convergence Rates}
\label{subsec:conv}
\begin{figure}[htbp]
\begin{center}
\subfigure[$L^1$ order of convergence: $L^1$-errors of the ion-density at time $t=2.0$ are plotted for 100, 200, 400 and 800 cells]{
\includegraphics[width=5.0in, height=2.5in]{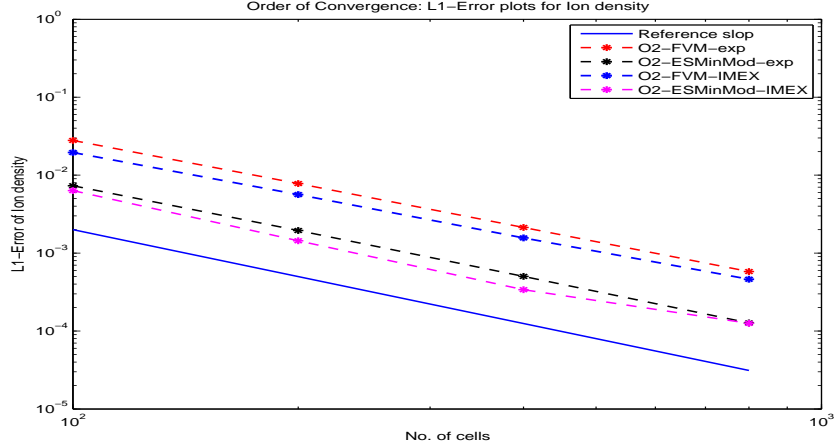}
\label{fig:l1_error}
}
\subfigure[Ion-density plots at time $t=2.0$ using 100 cells]{
\includegraphics[width=5.0in, height=2.5in]{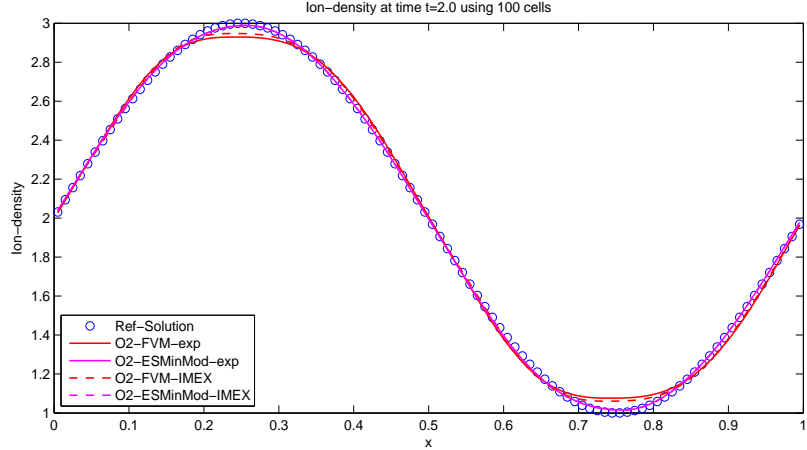}
\label{fig:exact_comp}
}
\caption{Errors of second order schemes}
\label{fig:error}
\end{center}
\end{figure}

As it is not possible to obtain explicit solution formulas for the two-fluid equations, we employ a forced solution approach to manufacture explicit solutions.  

In one space dimension, we consider the \emph{modified} equation:
$$
\pa_t\uc +\pa_x \f(\uc)=\sou(\uc) + K(x,t).
$$
with forcing term:
$$
K(x,t)=\{{\bf 0}_{13},-(2+\sin(2 \pi (x-t))),0,0,2+\sin(2 \pi (x-t)),0 \}^{\top}
$$ 
The initial densities are $\rhoi=\rhoe=2.0+\sin(2\pi x),$ with the velocities $v_i^x=v_e^x=1.0$ and the pressures $p_i=p_e=1.0.$ The initial magnetic field is $B^y=\sin(2\pi x)$ and the electric field is $E^z=-\sin(2\pi x).$ The computational domain is $(0,1)$ with periodic boundary conditions. The ion-electron mass ratio is set to $m_i/m_e=2.0$. 

It is straightforward to check that the exact solution is 
\begin{eqnarray*}
\rho_i=\rho_e=2.0+\sin(2 \pi (x-t)).
\end{eqnarray*}
In Figure \ref{fig:l1_error}, we have plotted the $L^1$ errors for the second-order schemes based on entropy stable fluxes with minmod (ES-MinMod) reconstruction for the spatial discretization and the second order SSP-RK scheme for time updated. For comparison, we have also plotted the results for the second-order FVM scheme based on a four wave HLL type solver with minmod limiter (O2-FVM). We observe that entropy stable schemes are significantly less diffusive than the standard FVM schemes. This is further verified by the solution plots in Figure \ref{fig:exact_comp}. The entropy stable version of the IMEX scheme is also less diffusive than its FVM counterpart. However, we observe that rate of convergence for the IMEX scheme falls when we refine the mesh. This is on account of splitting errors (in each RK2 sub-step) for the IMEX schemes.  
%%--------------------------------------------------------------------------
\subsection{Soliton Propagation in One Dimension}
\label{subsec:sol1d} 
Soliton propagation in two-fluid plasmas are simulated in \cite{hakim06,sol1,sol2,sol3}. It is shown that ion-acoustic solitons can form from an initial density hump.  In this section, we follow \cite{hakim06,sol2}, to simulate solitons in one space dimension. 

\begin{figure}[htbp]
\begin{center}
\subfigure[Ion-density evolution: Ion-density at non-dimensional $t=1,2,3,4,$ and $5$ for various schemes]{
\includegraphics[width=5.0in, height=2.5in]{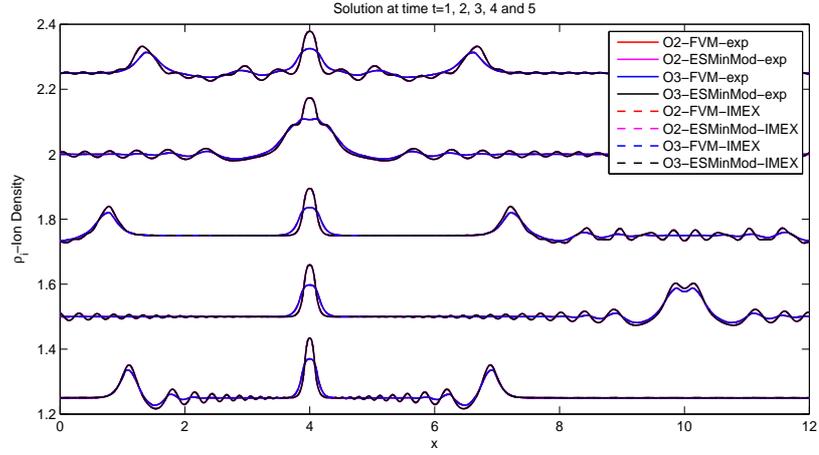}
\label{fig:s1_time}
}
\subfigure[Ion-density at non-dimensional time $t=5.0$ for various schemes]{
\includegraphics[width=5.0in, height=2.5in]{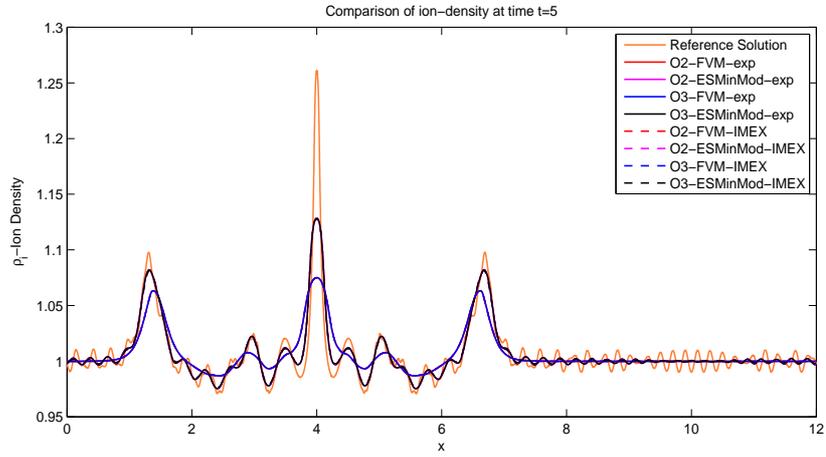}
\label{fig:s1_comp}
}
\subfigure[Ion-density at non-dimensional time $t=5.0$ for various schemes: Zoomed at $x=1.35$]{
\includegraphics[width=5.0in, height=2.5in]{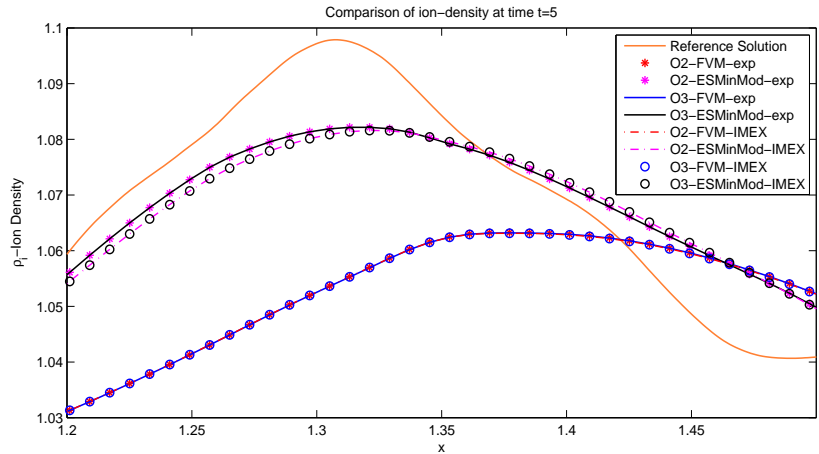}
\label{fig:s1_comp_l}
}
\caption{Soliton propagation using 1500 cells and Larmor radius $\rg=10^{-2}$.}
\label{fig:s1}
\end{center}
\end{figure}

\begin{figure}[htbp]
\begin{center}
\subfigure[Ion-density evolution: Ion-density at non-dimensional $t=1,2,3,4,$ and $5$ for various schemes]{
\includegraphics[width=5.0in, height=2.5in]{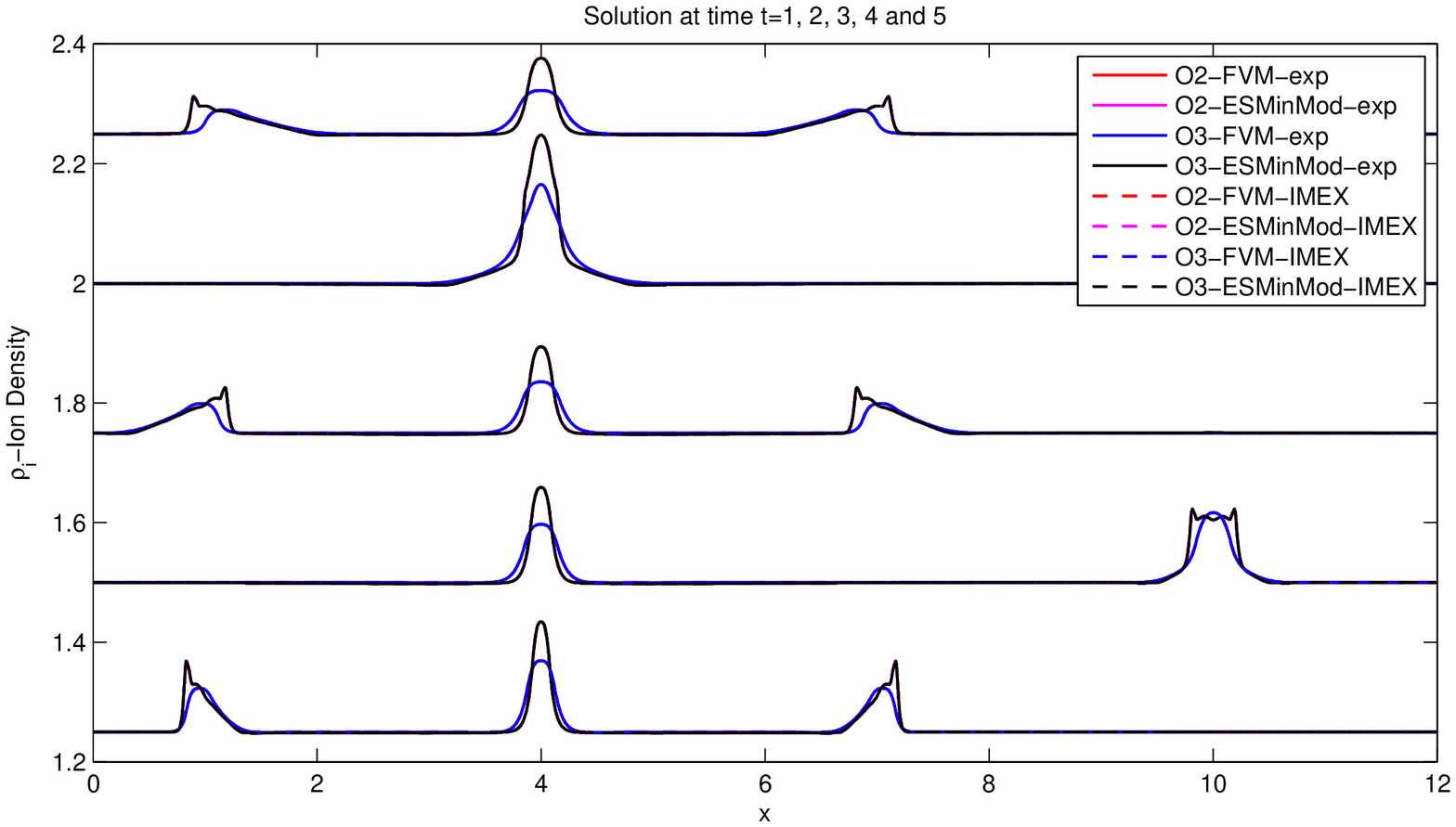}
\label{fig:s2_time}
}
\subfigure[Ion-density at non-dimensional time $t=5.0$ for various schemes]{
\includegraphics[width=5.0in, height=2.5in]{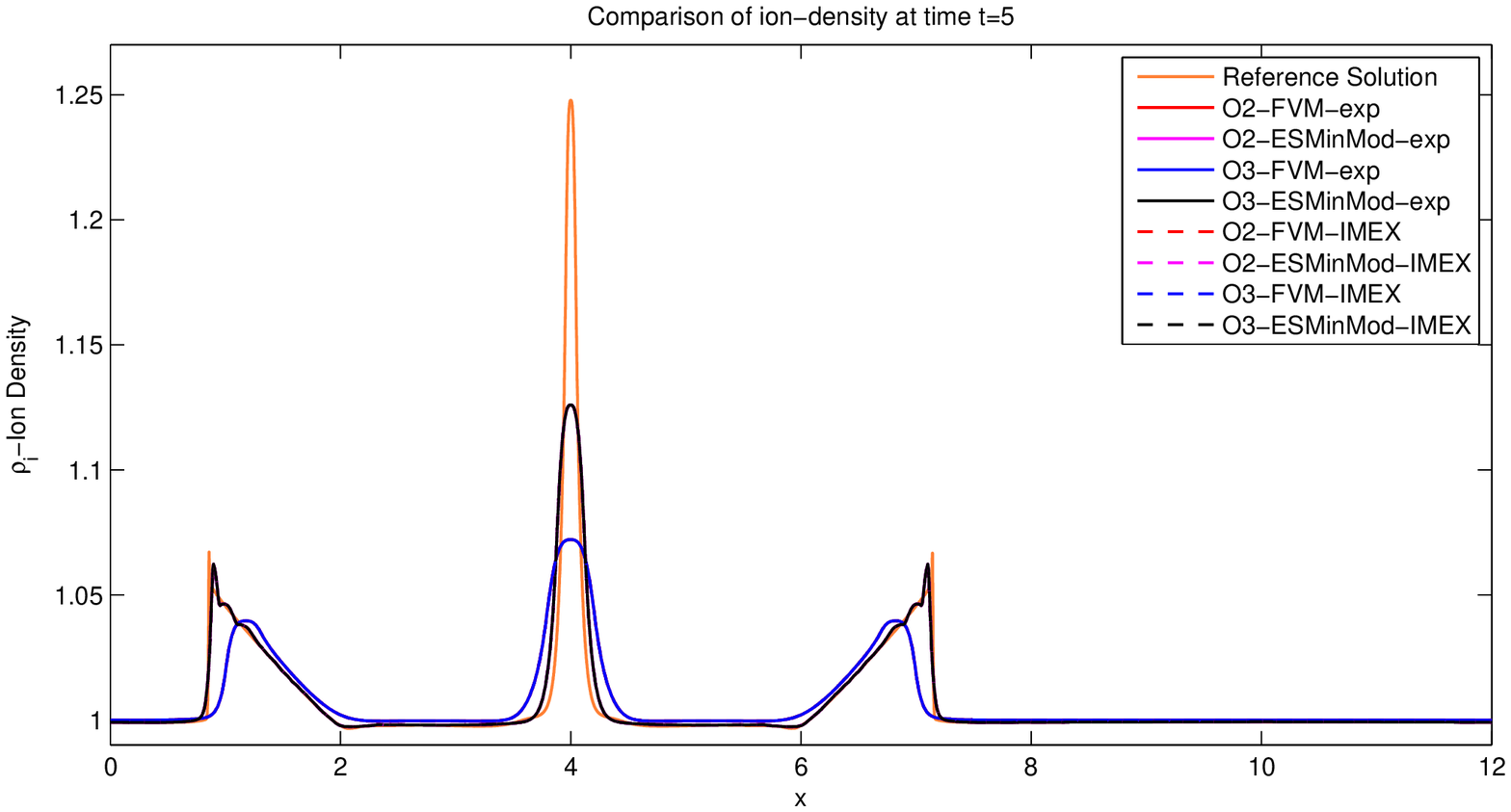}
\label{fig:s2_comp}
}
\caption{Soliton propagation using 1500 cells and Larmor radius $\rg=10^{-4}$.}
\label{fig:s2}
\end{center}
\end{figure}

\begin{figure}[htbp]
\begin{center}
\subfigure[Ion-density evolution: Ion-density at non-dimensional $t=1,2,3,4,$ and $5$ for various schemes]{
\includegraphics[width=5.0in, height=2.5in]{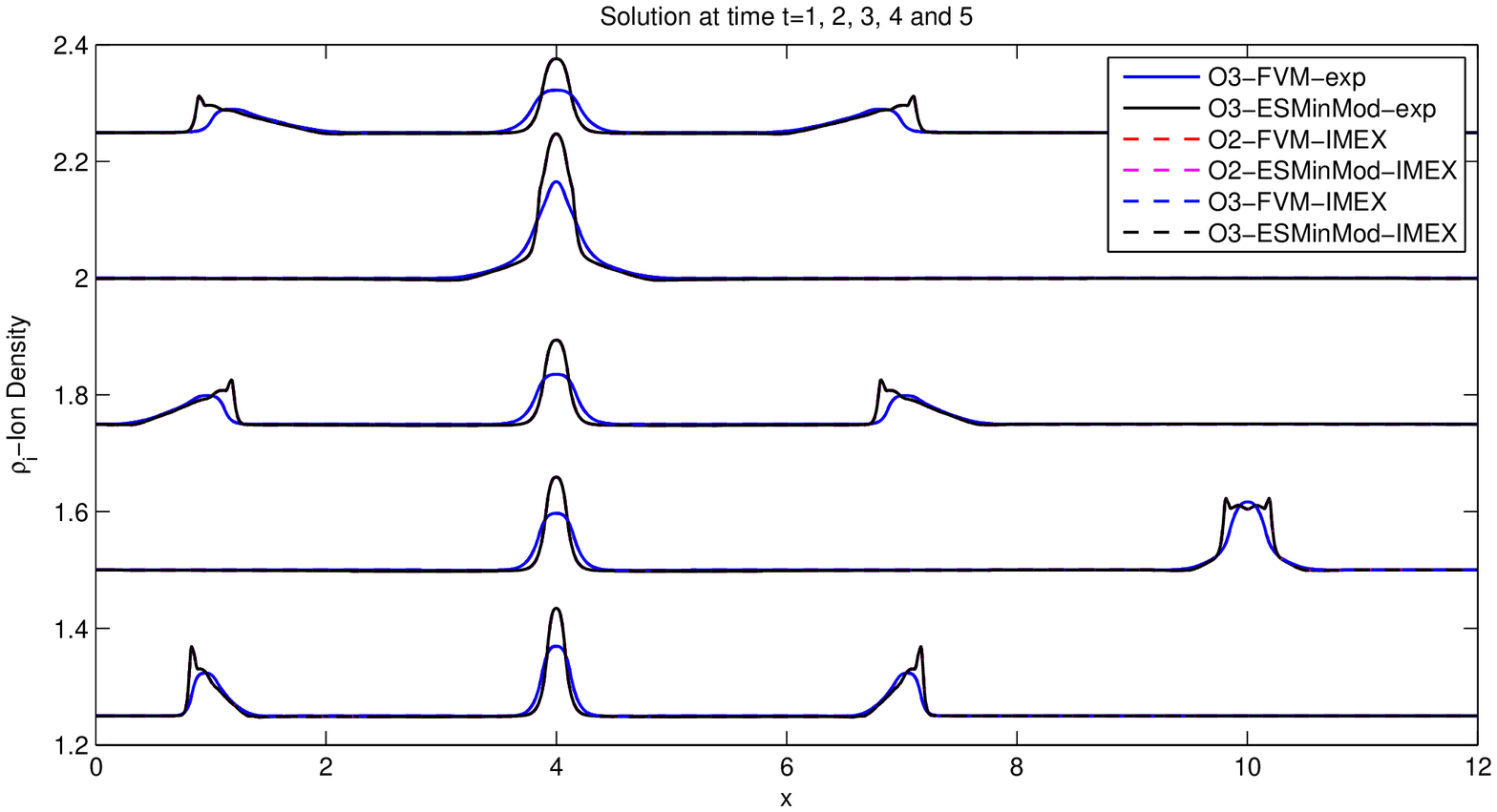}
\label{fig:s3_time}
}
\subfigure[Ion-density at non-dimensional time $t=5.0$ for various schemes]{
\includegraphics[width=5.0in, height=2.5in]{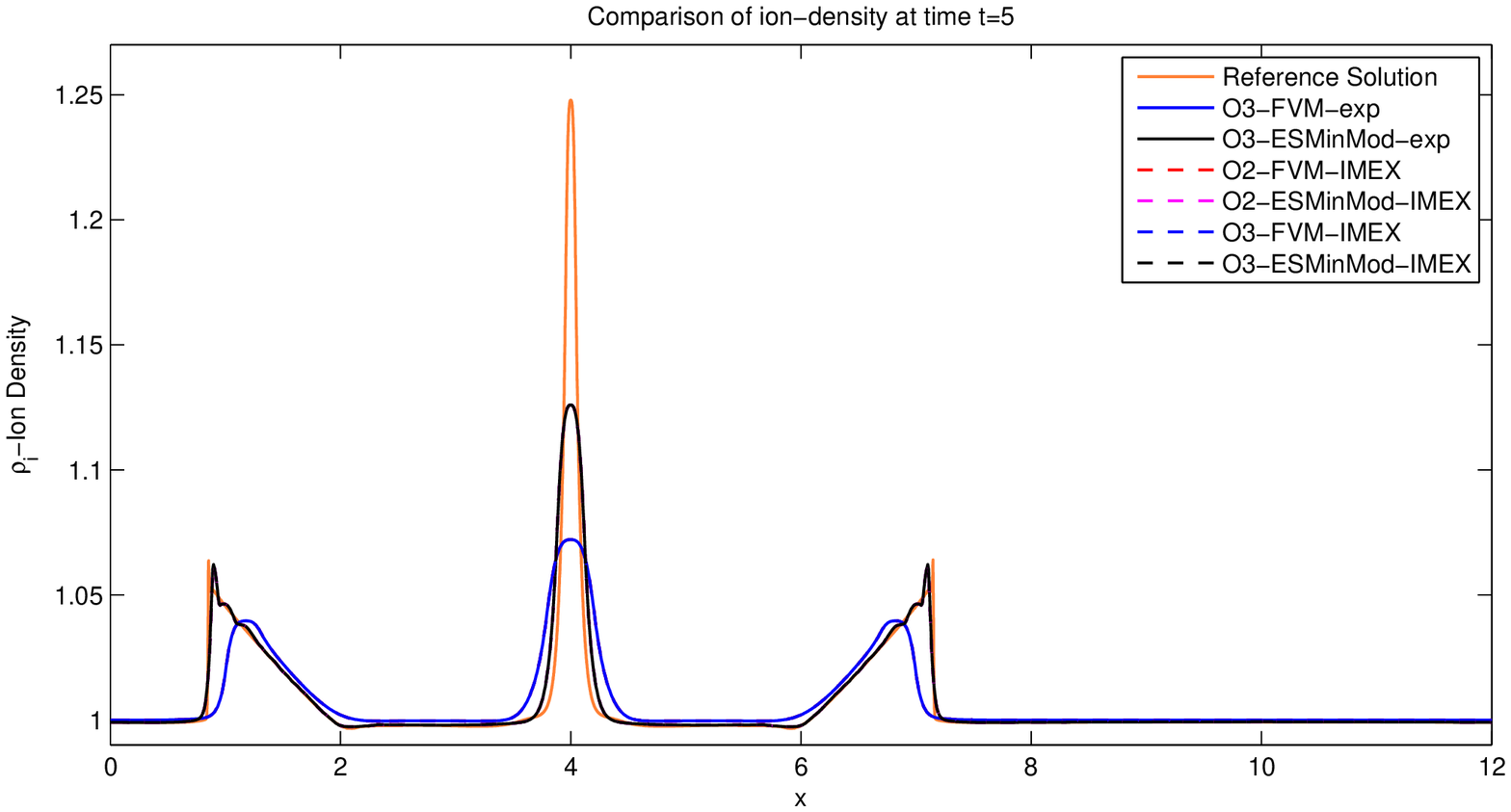}
\label{fig:s3_comp}
}
\caption{Soliton propagation using 1500 cells and Larmor radius $\rg=10^{-6}$.}
\label{fig:s3}
\end{center}
\end{figure}
Initially, the plasma is assumed to be stationary with ion density,
 \be
 \label{eq:sol}
 \rhoi = 1.0+ \exp(-25.0|x-L/3.0|),
\ee
and mass ratio $m_i/m_e=25,$, on the computational domain $D=(0,L)$ with $L=12.0$. Electron pressure is $p_e=5.0\rhoi$ with an ion-electron pressure ratio of $1/100$. Normalized Debye length is taken to be 1.0. Periodic boundary conditions are used. We consider three different Larmor radii: $10^{-2},\;10^{-4}$ and $10^{-6}$. Numerical solutions are computed using 1500 cells. The simulations are carried out using an MPI parallelized version of the code, on 10 computational cores.

The solutions are plotted for second order, spatially accurate entropy stable schemes (ESMN), using second (O2-ESMN) and third order (O3-ESMN) SSP Runge -Kutta, explicit and IMEX time stepping routines. We have also plotted the corresponding FVM solutions. The reference solutions for these simulations are computed using the O3-ESMN-IMEX scheme on 10000 mesh points. 

In Figure \ref{fig:s1}, we have plotted solutions corresponding to the Larmor radius of $10^{-2}.$ This corresponds to the simulation performed in \cite{hakim06}.  In Figure \ref{fig:s1_time}, we have plotted the ion-density profile at non-dimensional times $t=1,2,3,4$ and $5$. We observe that all the schemes are able to capture soliton waves.  In particular, the speed of soliton propagation is the same for all the schemes. In Figure. \ref{fig:s1_comp}, we have plotted the solutions at non-dimensional time $t=5.0$ and compared them with the reference solution. We again observe that the entropy stable schemes are more accurate than the corresponding FVM schemes. However it is hard to distinguish between some schemes in Figure. \ref{fig:s1_comp}, as solution lines for O2-FVM-exp, O3-FVM-exp, O2-FVM-IMEX and O3-FVM-IMEX coincide. Similarly, solution lines for O2-ESMN-exp, O3-ESMN-exp, O2-ESMN-IMEX and O3-ESMN-IMEX lie on top of each other in Figure. \ref{fig:s1_comp}. To, further analyze the schemes in Figure. \ref{fig:s1_comp_l}, we have zoomed in on the solution at $x=1.35$.  We notice that ESMN-IMEX schemes are slightly more diffusive than the ESMN-exp schemes. 

Compared to the solutions presented in \cite{hakim06}, entropy stable schemes appear to be more diffusive. However, in \cite{hakim06} authors have used a fourth order Runge-Kutta update for the source updates. Additionally, observe that both entropy stable schemes and wave propagation method fails to capture the oscillation around $x=10.0$ at the low resolution of 1500 cells. These oscillations are present in the solution only at finer resolutions. 

In Figures \ref{fig:s2} and \ref{fig:s3}, we have plotted the solutions for Larmor radii of $10^{-4}$ and $10^{-6}$, respectively. In Figures \ref{fig:s2_time} and \ref{fig:s3_time}, we have plotted the ion-density at non-dimensional times $t=1,2,3,4$ and $5$. As in the previous case, we observe that all schemes capture soliton waves. Furthermore, from the solution plots at non-dimensional time $t=5.0,$ in Figures \ref{fig:s2_comp} and \ref{fig:s3_comp}, we again note that the entropy stable schemes are less diffusive than FVM schemes. For the case of Larmor radius $10^{-6}$, we have not presented the solution for second order explicit time updates due to the very large simulation times, required for these schemes.
\begin{center}
\begin{table}[htbp]
\begin{tabular}{|c|c|c|c|}
\hline Scheme  & $\rg=10^{-2}$ & $\rg=10^{-4}$ & $\rg=10^{-6}$\\
%\hline O2-FVM-exp  & 18.46  & 1217.6 &-\\
\hline O2-ESMinMod-exp & 100.42 & 5089.67 &-\\
%\hline O3-FVM-exp & 27.77 & 111.88 & 12004.3\\
\hline O3-ESMinMod-exp & 152.26  & 533.85 & 74159.3 \\
%\hline O2-FVM-IMEX  & 20.34 & 19.44 & 23.09\\
\hline O2-ESMinMod-IMEX & 103.67 & 106.53 & 102.87\\
%\hline O3-FVM-IMEX & 29.73 & 33.64 & 28.74\\
\hline O3-ESMinMod-IMEX & 151.83 & 152.3 & 151.71 \\
\hline
\end{tabular} 
\caption{Comparison of simulation times of the numerical schemes for Larmor radii of $10^{-2},\;10^{-4}$ and $10^{-6}$ using 1500 cells.}
\label{tab:simu_time1d}
\end{table}
\end{center}

The above figures show that the IMEX schemes are slightly more diffusive than the explicit schemes for the same resolution and for the same spatial discretization. A natural question that arises in this context is why should be IMEX schemes be used when they only differ marginally in resolution with the explicit time stepping schemes ?. The answer to this lies in the computational run-time. 

From the source term for the two-fluid equations \eqref{eq:flux_source}, we see that the Larmor radius is a crucial parameter in determining the strength of the source term. Reducing the Larmor radius leads to an increase in the strength (and hence, stiffness) of the source term. Furthermore, the Larmor radius does not determine the speed of the waves in the two-fluid system. Hence, reducing the Larmor radius is a good test for evaluating the relative advantage of IMEX schemes over explicit time marching schemes.

To this end, we consider soliton propagation with different Larmor radii of $10^{-2},10^{-4}$ and $10^{-6}$, respectively. As the Larmor radius does not influence the wave speed, the time step for the IMEX schemes remains similar for the three simulations (with different Larmor radii). On the other hand, the increase in the strength of the source term, due to the decrease in the Larmor radius, implies a reduction in the time step for an explicit scheme. Therefore, we expect to see a difference in the computational cost between the implicit and explicit schemes. 

The simulation run-time for the three simulations (with different Larmor radii), on a mesh of $1500$ points, with all other simulation parameters being constant, are shown in Table \ref{tab:simu_time1d}. The table shows that the runtime for explicit schemes increases dramatically as the Larmor radius is reduced. The second-order scheme is particularly affected as the stability region for RK2 is quite small and it requires smaller time steps. In fact, for the Larmor radius of $10^{-4}$, the second-order (in time) explicit scheme is about $10$ times slower than the third-order (in time) explicit scheme. As a consequence, the run-time for the second-order explicit scheme on a Larmor-radius of $10^{-6}$ is too large and the run was not completed. The run-time for the third-order explicit scheme was also very large. On the other hand, the time taken by the implicit schemes (for both second- and third-order time stepping) is constant with respect to the Larmor radius. This implies a \emph{massive speed up} of the IMEX schemes (about a factor of $500$) when compared to the explicit schemes. This example illustrates the main advantage of the IMEX schemes: their robustness with respect to very low Larmor radii.

%--------------------------------------------------------------------------
\subsection{Generalized Brio-Wu Shock tube Problem}
\label{subsec:brio}
\begin{figure}[htbp]
\begin{center}
\subfigure[Generalized Brio-Wu shock tube problem: 1000 cells were used. Numerical solutions are plotted for $\rg=10.0$]{
\includegraphics[width=5.0in, height=2.5in]{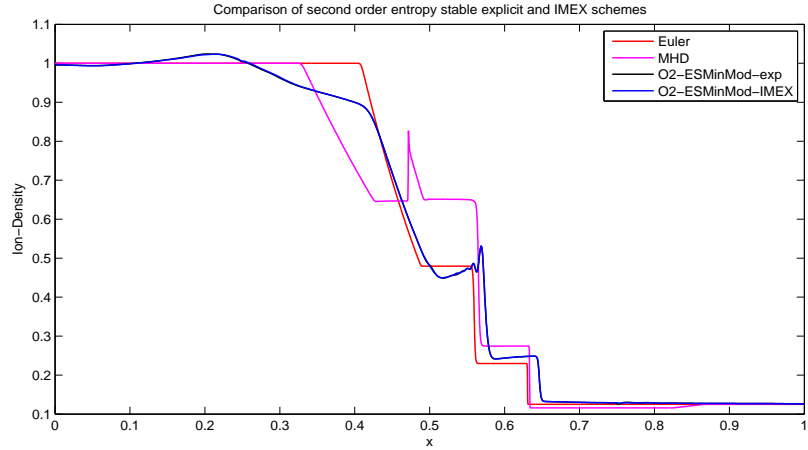}
\label{fig:tf1}
}
\subfigure[Generalized Brio-Wu shock tube problem: 50000 cells were used. Numerical solutions are plotted for second order schemes with $\rg=0.001$]{
\includegraphics[width=5.0in, height=2.5in]{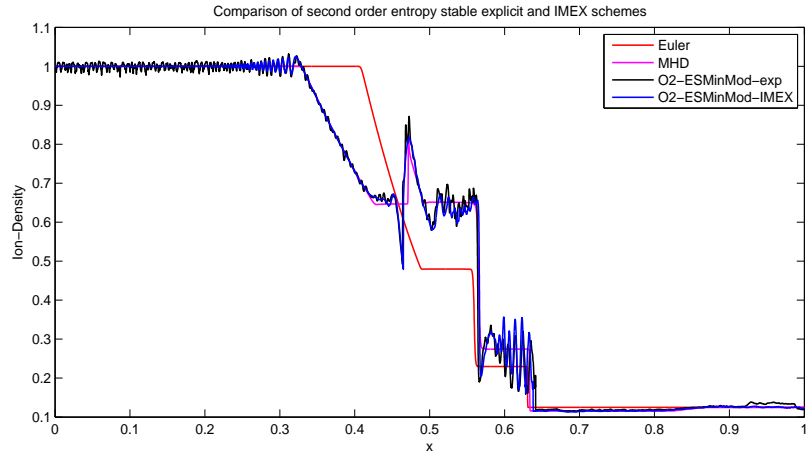}
\label{fig:tf4o2}
}
\subfigure[Generalized Brio-Wu shock tube problem: 50000 cells were used. Numerical solutions are plotted for second order schemes with $\rg=0.001$]{
\includegraphics[width=5.0in, height=2.5in]{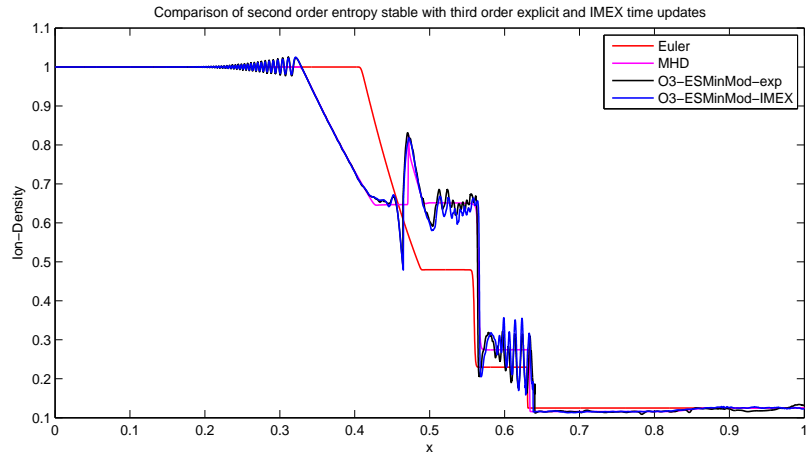}
\label{fig:tf4o3}
}
\caption{Generalized Brio-Wu shock tube Riemann problem}
\label{fig:tfbw}
\end{center}
\end{figure}
The initial conditions for this Riemann problem are
\begin{equation}
\label{eq:tfbw}
\U_{\mbox{left}}=\begin{cases}
    \rho_i = 1.0\\
  p_i =5 \;\times\;10^{-5}\\
  \rho_e = 1.0\;\; {m_e}/{m_i}\\
  p_e = 5 \;\times\;10^{-5}\\
  B^x = 0.75\\
  B^y = 1.0\\
  \vi=\ve=\E=0\\
  \phi=\psi=B^z=0
    \end{cases} \qquad
\U_{\mbox{right}}=\begin{cases}
  \rho_i = 0.125\\
  p_i =5 \;\times\;10^{-6}\\
  \rho_e = 0.125\;\; {m_e}/{m_i}\\
  p_e = 5 \;\times\;10^{-6}\\
  B^x = 0.75\\
  B^y = -1.0\\
  \vi=\ve=\E=0\\
  \phi=\psi=B^z=0
  \end{cases}
\end{equation}
on the computational domain $(0,1)$ with, $\U=\U_{\mbox{left}}$ for $x<0.5$ and $\U=\U_{\mbox{right}}$ for $x>0.5.$ The ion-electron mass ratio is taken to be ${m_i}/{m_e}=1836$. The initial conditions are non-dimensionalized using  $p_0=10^{-4}$. Non-dimensional Debye length is taken to be $0.01$. Simulations are carried out using Larmor radii of 10 and 0.001. Neumann boundary conditions are used. 

The purpose of this numerical experiment is to demonstrate the behavior of the solutions of two-fluid equations in two different regimes: one with high Larmor-radius and another with very low Larmor radius, respectively.

Numerical solutions are presented in Figure \ref{fig:tfbw}. In Figure \ref{fig:tf1}, we have plotted the numerical solutions based on O2-ESMinMod scheme using second order explicit and IMEX time updates. Solutions are computed with non-dimensional Larmor radius of 10.0, using 1000 cells. We observe that solution is very close to the solution of the Euler equations for each species. Note that letting $\hat{r}_g \rightarrow \infty$, one recovers the uncoupled equations of gas dynamics for both species. Furthermore, both IMEX and explicit schemes produce very similar results. 

The second regime that we investigate is to set the Larmor radius to $10^{-3}$. One expects to recover the MHD limit for vanishing Larmor radius. This limit is quite complicated to compute as one has to resolve the small-scale Langmuir oscillations, necessitating very fine meshes (see \cite{hakim06}). We show results obtained on a mesh of $50000$ cells both for second-order and third-order (in time) entropy stable (explicit as well as IMEX) schemes in figures \ref{fig:tf4o2} and \ref{fig:tf4o3}, respectively. 
 
We observe that the both explicit and IMEX solutions are converging to the MHD limit. However the second-order (in time) explicit scheme produces some small scale oscillations (near the left boundary). These small scale oscillations are not present in the results present in \cite{hakim06} as the source term in \cite{hakim06} is discretized using a fourth order Runge-Kutta update. On the other hand, the IMEX schemes resolves all the waves correctly. For the explicit schemes, small scale oscillations disappear when SSP-RK3 time update is used (see Figure \ref{fig:tf4o3}) and the results are comparable to those present in \cite{hakim06} in this case.
%%--------------------------------------------------------------------------
\subsection{Soliton Propagation in Two space dimensions}
\label{subsec:sol2d}

\begin{figure}[htbp]
\begin{center}
\subfigure[Ion-density evolution: Ion density $\rhoi$ at time $t=0.0,0.1,0.2$ and $0.3.$]{
\includegraphics[width=4.0in, height=2.5in]{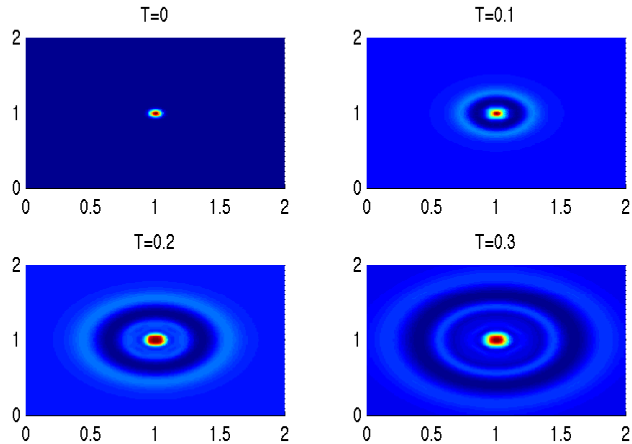}
\label{fig:s2d1_time}
}
\subfigure[Comparison of the schemes: Ion-density cut at $x=1$ for time $t=0.3$ of entropy stable schemes using third order explicit and IMEX time update.]{
\includegraphics[width=4.0in, height=2.5in]{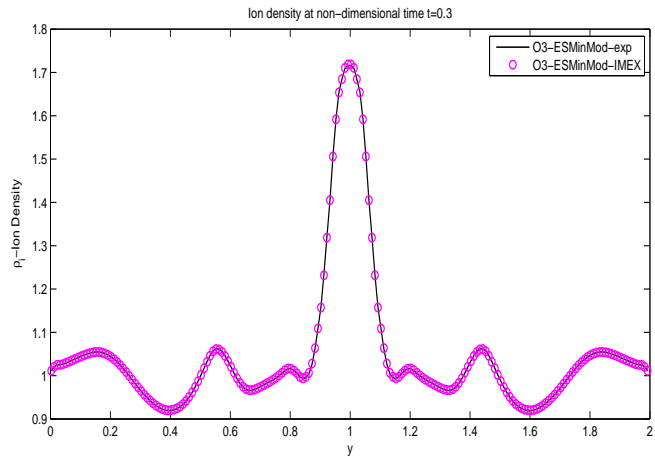}
\label{fig:s2d1_comp}
}
\caption{Soliton propagation in two dimensions on $200\times200$ mesh with $\rg=10^{-2}$.}
\label{fig:s2d1}
\end{center}
\end{figure}

\begin{figure}[htbp]
\begin{center}
\subfigure[Ion-density evolution: Ion density $\rhoi$ at time $t=0.0,0.1,0.2$ and $0.3.$]{
\includegraphics[width=4.0in, height=2.5in]{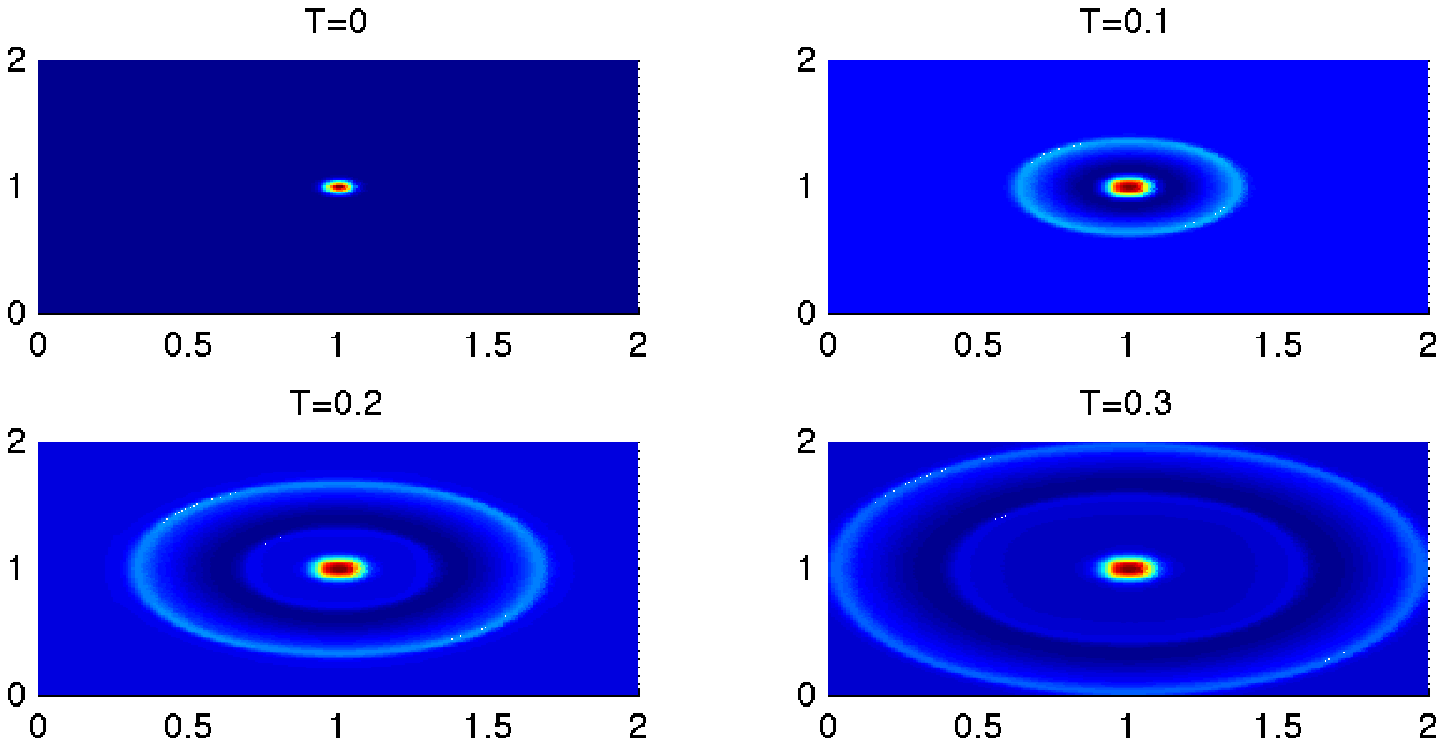}
\label{fig:s2d2_time}
}
\subfigure[Comparison of the schemes: Ion-density cut at $x=1$ for time $t=0.3$ of entropy stable schemes using third order explicit and IMEX time update.]{
\includegraphics[width=4.0in, height=2.5in]{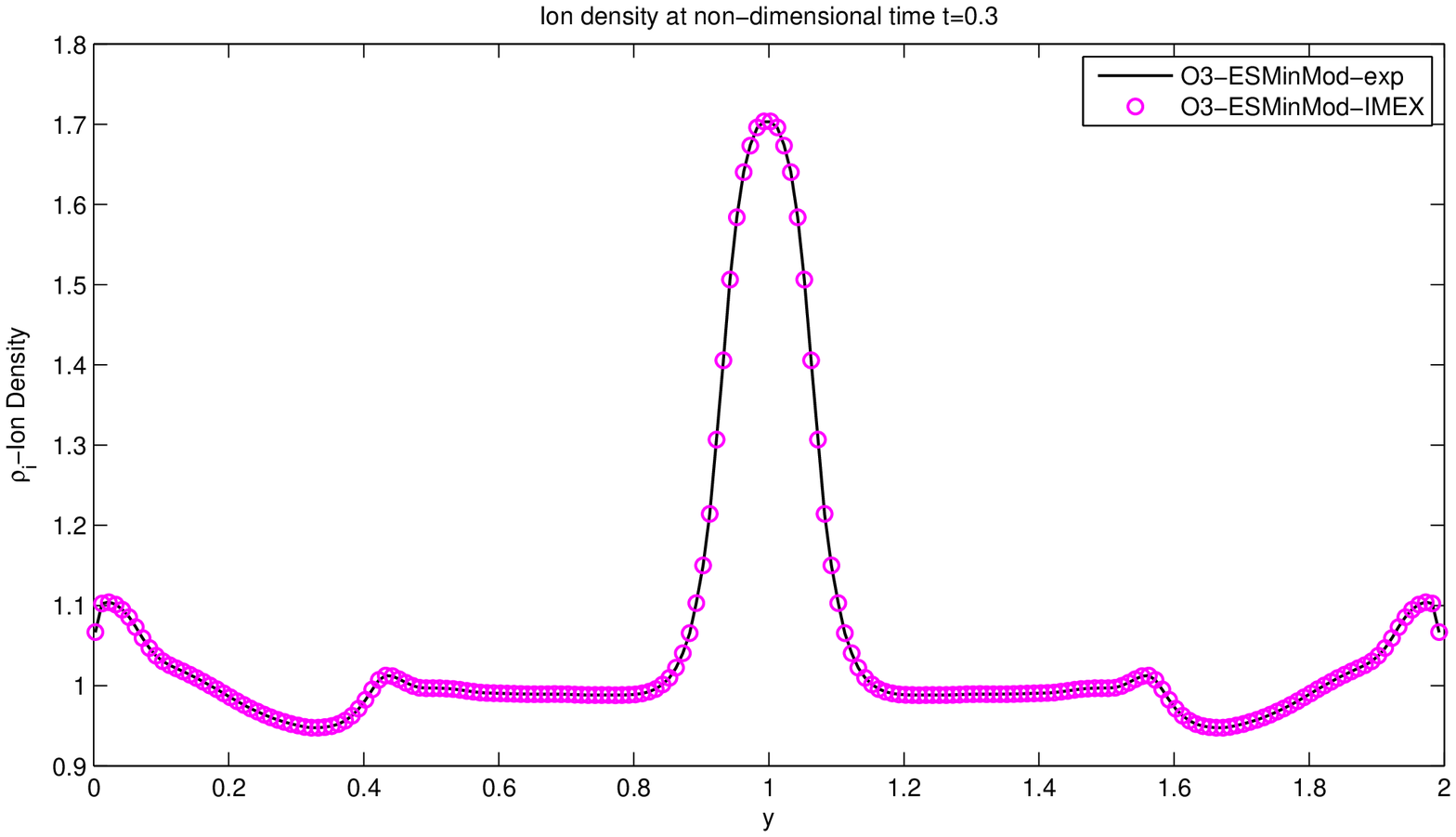}
\label{fig:s2d2_comp}
}
\caption{Soliton propagation in two dimensions on $200\times200$ mesh with $\rg=10^{-4}$.}
\label{fig:s2d2}
\end{center}
\end{figure}

Two dimensional soliton simulations were presented in \cite{sol3}. We follow \cite{sol3} and simulate 2-d solitons by considering the initial ion-density to be 
\be
\rhoi = 1.0+ 5.0\exp(-500.0( (x- L_x/2.0 )^2 + (y- L_y/2.0 )^2 ) 
\label{eq:sol_2d}
\ee
on the computational domain $(0,L_x)\times(0,L_y)$ with $L_x=L_y=2.0$. All other initial conditions are same as in the case of one dimensional soliton propagation in section \ref{subsec:sol1d}. Neumann boundary conditions are used to allow the waves to exit the domain without noticeable reflections. Note that we consider the ion-electron mass ratio of $25$ as compared to the ratio of $10$, considered in \cite{sol3}. Furthermore, we use Larmor radii of $10^{-2}$ and $10^{-4}$, compared to $10^{-1}$, used in \cite{sol3}. We expect dispersive waves moving outwards, similar to the one dimensional case, considered in section \ref{subsec:sol1d} (Also see\cite{sol3}).

\begin{center}
\begin{table}[htbp]
\begin{tabular}{|c|c|c|c|}
\hline Scheme  & $\rg=10^{-2}$ & $\rg=10^{-4}$\\
\hline O3-ESMinMod-exp & 907.2  & 2661.36  \\
\hline O3-ESMinMod-IMEX & 921.82 & 939.96  \\
\hline
\end{tabular} 
\caption{Comparison of simulation times of the numerical schemes for Larmor radii of $10^{-2}$ and $10^{-4}$, using $200\times 200$ cells.}
\label{tab:simu_time2d}
\end{table}\end{center}

Numerical results are presented in Figures \ref{fig:s2d1} and \ref{fig:s2d2}, corresponding to the Larmor radii of $10^{-2}$ and $10^{-4}$, respectively. In Figure \ref{fig:s2d1_time} and \ref{fig:s2d2_time} we have plotted the solution at non-dimensional time of $t=0,0.1,0.2$ and $0.3$ using O3-ESMN-IMEX scheme. The wave structure observed is similar to the one dimensional case. In Figure \ref{fig:s2d1_comp} and \ref{fig:s2d2_comp}, we have  plotted one dimensional cuts of the solution at  $x=1$ and at non-dimensional time $t=0.5$ for O3-ESMN-exp and O3-ESMN-IMEX schemes. As seen in the figures, the initial density hump bearks into a standing wave, centered at the origin, together with dispersive waves that propagate outward. We observe similar performances for both schemes. Furthermore, the IMEX scheme is faster than the explicit scheme for the low Larmor radius simulation. 
\section{Conclusion}
\label{sec:con}
We consider the two-fluid plasma equations and design finite difference schemes to approximate them. The semi-discrete version of the scheme is shown to be entropy stable. As the source terms in the two-fluid equations can be stiff, we propose IMEX schemes that treat the source terms implicitly. The novelty of our approach, in this context, is to observe that the special structure of the two-fluid plasma equations allows us to write the implicit (in source) time update as a local (in each cell) linear system of equations. This system can be solved exactly. Hence, our IMEX scheme does not require any Newton iterations or linearizations. 

Both the explicit and IMEX entropy stable schemes are shown to perform robustly on a set of numerical experiments. The entropy stable schemes are clearly more accurate than standard HLL type finite volume schemes. The main advantage of the IMEX schemes lie in the fact that they are robust (in run-time) with respect to a decrease in the Larmor radius. In particular, on (realistic) low Larmor radii simulations, the IMEX schemes can gain orders of magnitude in speedup as compared to the explicit schemes.

We will extend the entropy stable schemes to even higher order of accuracy and couple them with adaptive mesh refinement to be able to simulate realistic two-dimensional examples like magnetic reconnection, in a forthcoming paper.
%%-------------------------------------------------------

\end{document}